\documentclass[final]{siamart171218}
\usepackage{amsmath,amsfonts,amssymb}
\usepackage{a4wide}
\usepackage{multirow}
\usepackage{nicefrac}
\usepackage[colorinlistoftodos,prependcaption,textsize=small]{todonotes}
\usepackage{subcaption}
\usepackage{tikz}
\usetikzlibrary{shapes,calc}
\newcommand{\TheTitle}{A stencil scaling approach for
accelerating matrix-free finite element implementations}
\newcommand{\ShortTitle}{Stencil scaling for finite elements}
\newcommand{\TheAuthors}{S. Bauer, D. Drzisga, M. Mohr, U. R\"ude, C. Waluga,
and B. Wohlmuth}

\newcommand{\changed}[1]{#1}

\headers{\ShortTitle}{\TheAuthors}

\title{\TheTitle%
\thanks{Submitted to the editors \today.
\funding{This work was partly supported  by the German Research Foundation
through the Priority Programme 1648 "Software for Exascale Computing"
(SPP\-EXA) and by WO671/11-1.}}}
\author{
S. Bauer\thanks{Dept.~of Earth and Environmental Sciences,
Ludwig-Maximilians-Universit{\"a}t M{\"u}nchen} \and
D. Drzisga\thanks{Institute for Numerical Mathematics (M2), Technische
Universit{\"a}t M{\"u}nchen} \and
M. Mohr\footnotemark[2] \and
U. R\"ude\thanks{Dept.~of Computer Science 10,
Friedrich-Alexander-Universit{\"a}t Erlangen-N{\"u}rnberg} \and
C. Waluga\footnotemark[3] \and
B. Wohlmuth\footnotemark[3]}
\date{\today}

\definecolor{Apricot}{RGB}{251,185,130}
\definecolor{SkyBlue}{RGB}{70,197,221}
\definecolor{YellowOrange}{RGB}{250,162,26}
\definecolor{ForestGreen}{RGB}{0,155,85}
\definecolor{RoyalBlue}{RGB}{0,113,188}

\newsiamremark{remark}{Remark}

\usepackage{xspace}

\newcommand{\RR}{\mathbb{R}}

\renewcommand{\div}{\mathop{\rm div}}

\newcommand{\stencildir}{w}
\newcommand{\integrald}[4]{\int_{#1}^{#2} #3 \,\mathrm{d}#4}

\newcommand{\xh}{\hat{x}}
\newcommand{\yh}{\hat{y}}
\newcommand{\zh}{\hat{z}}

\begin{document}
\maketitle
\begin{abstract}
We present a novel approach to fast on-the-fly low order finite element assembly for scalar elliptic partial differential equations of Darcy type with variable coefficients optimized for matrix-free implementations.
Our approach introduces a new operator that is obtained by appropriately scaling the reference stiffness matrix from the constant coefficient case.
Assuming sufficient regularity, an a priori analysis shows 
that solutions obtained by this approach are unique and have
asymptotically optimal order convergence in the $H^1$- and the $L^2$-norm
on hierarchical hybrid grids.
For the pre-asymptotic regime, we present a local modification that guarantees uniform ellipticity of the operator.
Cost considerations show that our novel approach requires roughly one third of the floating-point operations compared to a classical finite element assembly scheme employing nodal integration.
Our theoretical considerations are illustrated by numerical tests 
that confirm the expectations with respect to accuracy and run-time.
A large scale application with more than a 
hundred billion ($1.6\cdot10^{11}$) 
degrees of freedom executed on 14\,310 compute cores
demonstrates the efficiency of the new scaling approach.
\end{abstract}
%
%
\begin{keywords}
matrix-free, finite-elements, variable coefficients, stencil scaling, variational crime analysis, optimal order a priori estimates
\end{keywords}

\begin{AMS}
65N15, 65N30, 65Y20
\end{AMS}
\section{Introduction}
\label{sec:intro}

Traditional finite element implementations are based on 
computing local element stiffness matrices, followed by a local-to-global assembly step,
resulting in a sparse matrix. However, 
the cost of storing the global stiffness matrix is significant. 
Even for scalar equations and
low order 3D tetrahedral elements, the stiffness matrix has, on average, fifteen entries per
row, and thus a standard sparse matrix format will require thirty times as much
storage for the matrix as for the solution vector.
This 
limits the size of the problems that can be tackled
and becomes the dominating cost factor since the sparse matrix must
be re-read from memory repeatedly when iterative solvers are applied.
On all current and future computing systems memory throughput and
memory access latency 
can determine the run-time more critically than the floating-point operations executed.
Furthermore, 
energy consumption has been identified as one of
the fundamental roadblocks in exa-scale computing.
In this cost metric, memory access is again 
more expensive than computation. 
Against the backdrop of this technological development, it has
become mandatory to develop numerical techniques that 
reduce memory traffic.
In the context of partial differential equations this is leading to a re-newed
interest in so-called {\em matrix-free} techniques 
and -- in some sense -- to a revival of techniques that are well-known in the
context of finite difference methods.

Matrix-free techniques are motivated 
from the observation that 
many iterative
solvers, e.g., Richardson iteration or Krylov subspace methods, require
only the result of multiplying the global system matrix with a vector, 
but not the matrix itself. 
The former can be computed by 
local operations in each element, 
avoiding to set up, store, and load the global stiffness matrix.
One of the first papers in this direction is \cite{Carey:1986:CANM}, which
describes the so called element-by-element approach (EBE), in which the global
matrix-vector-product (MVP) is assembled from the contributions of MVPs of
element matrices with local vectors of unknowns. The element matrices are
either pre-computed and stored or recomputed on-the-fly.
Note that storing all element matrices, even for low-order elements, has a
higher memory footprint than the global matrix\footnote{However, it requires
less memory than storage schemes for sparse direct solvers which reserve space
for fill-in, the original competing scenario in \cite{Carey:1986:CANM}.} itself.

Consequently, the traditional EBE has not found wide application in unstructured
mesh approaches. However, it has been 
successfully applied in cases where
the discretization is based on undeformed hexahedral elements with tri-linear
trial functions, see e.g.~\cite{Arbenz:2008:IJMME,Bielak:2005:CMES,Flaig:2012:LSSC,Rietbergen:1996:IJMME}.
In such a setting, the element matrix is the same for all elements, which
significantly reduces the storage requirements, and variable material parameters
can be introduced by weighting local matrix entries in the same on-the-fly
fashion as will be developed in this paper.


Matrix-free approaches for higher-order elements, as described in
e.g.,~\cite{Brown:2010:JSC,Kronbichler:2012:CAF,Ljungkvist:2017:HPC17,Ljungkvist:2017:TechRep,May:2015:CMAME},
differ from the classic EBE approach in that they do not setup the element
matrix and 
consecutively multiply it with the local vector.
Instead, they fuse the two steps
by going back to numerical integration of the weak form itself. The process is
accelerated by pre-computing and storing certain quantities, such as
e.g., derivatives of basis functions at quadrature points within a reference
element. These techniques \changed{in principle} work for arbitrarily shaped elements and orders,
although a significant reduction of complexity can be achieved for tensor-product elements.

However, these matrix-free approaches have also shortcomings.
While the low-order settings
\cite{Arbenz:2008:IJMME,Bielak:2005:CMES,Flaig:2012:LSSC,Rietbergen:1996:IJMME}
require structured hexahedral meshes, modern techniques for unstructured meshes
only pay off for elements \changed{with tensor-product spaces} with polynomial orders of at least two; see
\cite{Kronbichler:2012:CAF,Ljungkvist:2017:HPC17}.

In this paper we will present a novel matrix-free approach for low-order
finite elements designed for the hierarchical hybrid grids framework
(HHG); see e.g.~\cite{Bergen:2006:SCS,bergen-huelsemann_2004,bergen2007hierarchical}.
HHG offers significantly more geometric flexibility than undeformed
hexahedral meshes. It is based on two interleaved ideas.
The first one is a special discretization
of the problem domain. In HHG, the computational grid is created by way of a
uniform refinement following
the rules of \cite{Bey95}, starting from a possibly unstructured simplicial
macro mesh.
\changed{The resulting nested hierarchy of meshes allows for the
implementation of powerful geometric multigrid solvers.}
The elements of \changed{the} macro mesh are called macro-elements and
the resulting sub-triangulation reflects a uniform structure within these macro-elements. 
The second aspect is based on the fact that each row of the global finite
element stiffness matrix can be considered as a difference stencil. This
notion and point of view is classical on structured grids and recently has
found re-newed interest in the context of finite elements too; see
e.g.,~\cite{Engwer:2017:CCPE}. In combination with the HHG grid construction
this implies that for linear simplicial elements one obtains stencils with
identical structure for each inner node of a macro primitive. \changed{We define macro primitives as the geometrical entities of the macro mesh of different dimensions, i.e.,}~vertex,
edge, face, and \changed{tetrahedrons}. If additionally the coefficients of the PDE are constant
per macro-element, then also the stencil entries are exactly the same.
Consequently, only a few different stencils (one per \changed{macro} primitive)
can occur and need to be stored. This leads to extremely efficient matrix-free
techniques, as has been demonstrated e.g.,~in
\cite{Bergen:2006:SCS,gmeiner2015towards}.

Let us now consider the setting of an elliptic PDE with piecewise smooth
variable coefficients, 
assuming that the macro mesh resolves 
jumps in the coefficients. 
In this case, a standard finite element formulation is based on 
quadrature formulas and introduces a variational crime. According to
\cite{CIA76,SF08}, there is flexibility 
how the integrals are approximated without degenerating the order of convergence. 
This has recently been exploited in \cite{Bauer:2017:Two-Scale} with a method that
approximates these integral values on-the-fly using \changed{suitable} surrogate
polynomials with respect to the macro mesh. The resulting  two-scale method
is able to preserve the convergence order if the coarse and the fine scale are related properly. Here we propose an alternative which is based on the fine scale.

For this article, we restrict ourselves to
the lowest order case of conforming finite elements on simplicial meshes. Then
the most popular quadrature formula is the one point Gauss rule which in the
simplest case of $\div (k \nabla u) $ as PDE operator just weights the
element based reference stiffness matrix of the Laplacian by the
factor of $k(x_T)$ where $x_T$ is the barycenter of the element $T$.
Alternatively, one can select a purely vertex-based quadrature formula. 
Here, the weighting of the element matrix is given by $\sum_{i=1}^{d+1} k(x_T^i)
/(d+1)$, where $d$ is the space dimension and $x_T^i$ are the vertices
of element $T$.
Using a vertex-based quadrature formula saves function evaluations and is,
thus, attractive whenever the evaluation of the coefficient function is
expensive and it pays off 
to reuse once computed values in several element stiffness matrices.
Note that reusing barycentric data
on general unstructured meshes
will require nontrivial storage schemes.

In the case of 
variable coefficient functions, stencil entries can vary from
one mesh node to another. The number of possibly different stencils within
each macro-element becomes $\frac{1}{d!}2^{d\ell} + \mathcal{O}(2^{(d-1)\ell})$,
where $\ell $ is the number of uniform refinement steps for HHG. Now we can
resort to two options: 
Either these stencils are computed once and then saved, effectively creating a
sparse matrix data structure, or they are computed on-the-fly each time when
they are needed. Neither of these techniques is ideal for extreme scale
computations.
While for the first option ${\mathcal O }(2^{d\ell})$ 
extra memory is consumed and extensive memory traffic occurs, the second option
requires re-computation of ${\mathcal O } (2^{d\ell}) $ local contributions.

The efficiency of a numerical PDE solver can be analyzed following
the {\em textbook paradigm}~\cite{brandt1998barriers}
that defines a work unit (WU) to be the cost of one application
of the discrete operator for a given problem.
With this definition, the analysis of iterative solvers can be conducted in terms of WU.
Classical multigrid textbook efficiency is achieved when 
the solution is obtained in less than 10 WU.
For devising an efficient method it is, however, equally critical to
design algorithms that reduce the cost of a WU without sacrificing accuracy.
Clearly, the real life cost of a WU depends on the computer hardware
and the efficiency of the implementation, 
as e.g.,~analyzed for parallel supercomputers in \cite{gmeiner2015towards}.
On the other side, matrix-free techniques, as the one proposed in this article,
seek opportunities to reduce
the cost of a WU by a clever rearrangement of the algorithms
or by exploiting approximations where this is possible; see e.g.,~also \cite{Bauer:2017:Two-Scale}.

These preliminary considerations motivate our novel approach
to reduce the cost of a WU by recomputing the surrogate stencil entries
for a matrix-free solver more efficiently.
We find that these values can be assembled from a reference stencil
of the constant coefficient case which is scaled appropriately using nodal
values of the coefficient function. 
We will show that under suitable conditions, this technique does not sacrifice accuracy.
However, we also demonstrate that the new method 
can reduce the cost of a WU considerably and in consequence helps to reduce the
time-to-solution.

The rest of this paper is structured as follows: In section 2, we define our new scaling approach. The variational crime is analyzed in Section 3 where optimal order a priori results for the $L^2 $- and $H^1$-norm are obtained. In section 4, we consider modifications in the pre-asymptotic regime to guarantee uniform ellipticity. Section 5 is devoted to the reproduction property and the primitive concept which allows for a fast on-the-fly reassembling in a matrix free software framework. In section 6, we discuss the cost compared to a standard nodal based element-wise assembling. Finally, in section 7 we perform numerically an accuracy study and a run-time comparison to illustrate the performance gain of the new scaling approach. 

\section{Problem setting and definition of the scaling approach}
\label{sec:frame}
We consider a scalar elliptic  partial differential equation of Darcy type, i.e.,
\begin{equation*}
-\text{div } K \nabla u = f, \quad \text{in } \Omega , \quad \text{tr } u = 0
\quad \text{on } \partial \Omega
\end{equation*}
where $\text{tr} $ stands for the boundary trace operator and $f \in L^2(\Omega)$.
Here $\Omega \subset \mathbb{R}^d$, $d=2,3$, is a bounded polygonal/polyhedral
domain, and $K$ denotes a uniformly positive and symmetric tensor
\changed{with coefficients specified through a number of functions}
$k_{m}$, $m=1 ,\ldots , M$,
where $M \leq 3$ in 2D  and $M \leq 6$ in 3D due to symmetry.

For the Darcy operator with a scalar uniform positive permeability, i.e.,  $-\div (k \nabla u)$,
we can set $M=1$ and  $k_1
:=k$. 
The above setting also covers blending finite elements approaches~\cite{GH73}.
Here $K$ is related to the Jacobian
of the blending function. For example, if the standard Laplacian model problem is considered on the physical domain $\Omega_{\text{phy}}$ but the actual assembly is carried out on a reference domain $\Omega := \Phi( \Omega_{\text{phy}})$,  we have
\begin{align}
a(v,w) = \integrald{\Omega_{\text{phy}}}{}{\nabla v_{\text{phy}} \cdot \nabla w_{\text{phy}}}{x_{\text{phy}}} = \integrald{\Omega}{}{\nabla  v \cdot \frac{(D \Phi)(D \Phi)^{\top}}{ | \text{det } D \Phi |} \nabla w}{x},
\label{eqn:aWithBlendingFunc}
\end{align} 
where $D\Phi $ is the Jacobian of the mapping $\Phi$, and $ v_{\text{phy}}  := v\circ \Phi $, $ w_{\text{phy}}  := w\circ \Phi $.

\subsection{Definition of our scaling approach}
\label{subsec:stencilScaling}
The weak form associated with the partial differential equation is defined in terms of
the bilinear form $a(v,w) := \integrald{\Omega}{}{\nabla v \cdot K \nabla w}{x}$, and the weak solution $u \in V_0 := H_0^1(\Omega)$ satisfies:
\begin{equation*}
a(u,v) = (f,v), \quad v \in V_0.
\end{equation*}
This bilinear form can be affinely decomposed as
\begin{align}
\label{eq:integral}
a(v,w) := \sum_{m=1}^{M} a_m (v,w), \quad
 a_m(v,w) := \integrald{\Omega}{}{k_m(x) (D_m v , D_mw)}{x}, \quad v, w \in V := H^1(\Omega)
\end{align}
where $D_m $ is a first order partial differential operator and
$(\cdot,\cdot)$ stands for some suitable inner product. In the case of a scalar permeability we find $D_1 := \nabla $ and $(\cdot,\cdot)$ stands for the scalar product in $\mathbb{R}^d$. \changed{
While for \eqref{eqn:aWithBlendingFunc} in 2D one
can, as one alternative, e.g.~define
\begin{gather*}
k_1:=(K_{11}-K_{12})
\enspace,\quad
k_2:=K_{12}
\enspace,\quad
k_3:=(K_{22}-K_{11})
\enspace,\\
D_1:=\nabla
\enspace,\quad
D_2:=\nicefrac{\partial}{\partial x}+\nicefrac{\partial}{\partial y}
\enspace,\quad
D_3:=\nicefrac{\partial}{\partial y}
\enspace,
\end{gather*}
%
%
where $K=(K_{ij})$ and the same scalar product $(\cdot,\cdot)$ as above.
Note that this decomposition reduces to the one in case of a scalar
permeability, i.e.~for $K=\text{diag}(k,k)$.}

Let ${\mathcal T}_H $, $H > 0 $ fixed, be a possibly unstructured
simplicial triangulation resolving 
$\Omega$. We call $ {\mathcal T}_H$ also macro-triangulation and denote
its elements by $T$.
Using  uniform mesh refinement, we obtain $ {\mathcal T}_{h/2}$ from $ {\mathcal T}_h$ by decomposing each
element into $2^d$ sub-elements, $h \in \{H/2, H/4, \ldots
\}$; see \cite{Bey95} for the  3D case. The elements of  $ {\mathcal T}_h$
are
denoted by $t$. The macro-triangulation is then decomposed into the following geometrical primitives: \emph{elements, faces, edges}, and \emph{vertices}. Each of these geometric primitives  acts as
a container for a subset of unknowns associated with the refined triangulations. These sets of unknowns can be stored in array-like
data structures, resulting in a contiguous memory layout that conforms inherently
to the refinement hierarchy; see \cite{bergen2007hierarchical,gmeiner2015towards}. In particular, the unknowns can be accessed without indirect
addressing such that the overhead is reduced significantly when compared to
conventional sparse matrix data structures.
Associated with   $ {\mathcal T}_h$ is the space $V_h \subset V$ of piecewise linear
finite elements. In $V_h$, we do not include the homogeneous boundary conditions.
We denote by $\phi_i \in V_h$  the nodal basis functions associated to the $i$-th mesh node. 
Node $i$ is located at the vertex $x_i$.
For $v_h := \sum_{i} \nu_{i} \phi_{i} $ and $w_h := \sum_{j} \chi_{j} \phi_{j}
$, we define our scaled discrete bilinear forms $a_h (\cdot,\cdot)$ and 
$a_m^h (\cdot,\cdot)$ by
\begin{subequations}
\label{eq:meshdep}
\begin{align} \label{eq:affinesplit}
a_h(v_h,w_h) &:= \sum_{m=1}^M a_m^h (v_h,w_h), \\ \label{eq:discrete}
 a_m^h(v_h,w_h) &:= \frac 14 \sum_{T \in {\mathcal T}_{H}}
 \sum_{i,j}  (k_m |_T (x_{i}) +  k_m |_T (x_{j}) )  (\nu_{i} - \nu_{j} ) (\chi_{j}  - \chi_{i}) \integrald{T}{}{(D_m \phi_i, D_m \phi_j)}{x}.
\end{align}
\end{subequations}
This definition is motivated by the fact that $ a_m(v_h,w_h)$ can be written as
\begin{align}
\label{eq:exakt}
 a_m(v_h,w_h) &= \frac 12 \sum_{T \in {\mathcal T}_{H}}
 \sum_{i,j}   (\nu_{i} - \nu_{j} ) (\chi_{j}  - \chi_{i}) \integrald{T}{}{k_m(x)  (D_m \phi_i, D_m \phi_j)}{x}.
\end{align}
Here we have exploited symmetry and the row sum property. It is obvious that if $k_m$ is a constant restricted to $T$, we do obtain
$a_m^h(v_h,w_h)=a_m(v_h,w_h)$. In general however, the
definition of $a_h(\cdot,\cdot)$ introduces a variational crime and it does not even correspond to an element-wise local assembling based on a quadrature formula. We note that each node  on $\partial T$ is redundantly existent in the data structure and that we can easily account for jumps in the coefficient function when resolved by the macro-mesh elements $T$.

Similar scaling techniques have been used in \cite{Yang:1997:PhD} for a generalized Stokes problem from \changed{geodynamics} with coupled velocity components.
However, for vectorial equations such a simple scaling does asymptotically not result in a physically correct solution.
For the computation of integrals on triangles containing derivatives in the form of (\ref{eq:integral}), cubature formulas of the form (\ref{eq:discrete}) in combination
with Euler-MacLaurin type asymptotic expansions have been applied
\cite[Table 1]{Lyness1998}.

\begin{remark}
At first glance the Definition \eqref{eq:discrete} might not be more attractive than \eqref{eq:exakt} regarding the computational cost. In a matrix free approach, however,
where we have to reassemble the entries in each matrix call,
\eqref{eq:discrete} turns out to be much more favorable. In order see this, we have to recall that we work with hybrid hierarchical meshes. This means that for each inner node $i$ in $T$, we find the same entries in the sense that
$$
 \integrald{T}{}{(D_m \phi_i, D_m \phi_j)}{x}  =  \integrald{T}{}{(D_m \phi_{l}, D_m \phi_{x_j + \delta x})}{x}.
 $$
Here we have identified the index notation with the vertex notation, and the vertex $x_l$ is obtained from the vertex $x_i$ by a shift of $\delta x$, i.e.,
$x_l = x_i + \delta x$. Consequently, the values of $ \integrald{T}{}{(D_m \phi_i, D_m \phi_j)}{x} $ do not have to be re-computed but can be efficiently stored.
\end{remark}

For simplicity of notation, we shall restrict
ourselves in the following to the case of the Darcy equation with a scalar uniformly positive definite permeability; i.e., $M=1$ and drop  the index $m$. However, the \changed{proofs in Sec.~\ref{sec:analysis}} can be generalized to conceptually the
 same type of results for $M > 1$. In Subsection~\ref{sub:generalM}, we also show numerical results for $M=6$ in 3D.

\subsection{Stencil structure}
\label{sec:stencil}
We exploit the hierarchical grid structure to  save a significant amount of
memory compared to classical sparse matrix formats.
%
Any direct neighbor $x_j \in \mathcal{N}_T(x_i)$ can be described through a direction
vector $\stencildir_j$ such that $x_j = x_i + \stencildir_j$. The regularity
of the grid in the interior of a macro-element $T$ implies that these vectors
remain the same, when we move from one node  to another node.
Additionally, for each neighbor $x_j \in \mathcal{N}_T(x_i)\setminus\{x_i\}$
there is a mirrored neighbor $x_j'$ of $x_i$ reachable by
$\stencildir_j = -\stencildir_{j'}$; see Fig.~\ref{fig:stencildir}.

Let $n_i = \left| \mathcal{N}_T(x_i) \right|$ denote the stencil size at mesh node $x_i$. We define the stencil
$\hat s_{x_i}^T \in \mathbb{R}^{n_i}$ associated to the i-th mesh node $x_i$  restricted on $T$ as
\begin{align*}
(\hat s_{x_i}^{T})_j
:= \integrald{T}{}{(\nabla \phi_{x_i + w_j}, \nabla \phi_{x_i})}{x}
\enspace.
\end{align*}
The symmetry of the bilinear form yields
\begin{align*}
(\hat s_{x_i}^{T})_j = (\hat s_{x_i + w_j}^{T})_{j'}.
\end{align*}
We recall that for each mesh node $x_i$ we have $n_i \leq 7$ in 2D and $n_i \leq 15$ in 3D.
Out of these entries only 3 in 2D and 7 in 3D have to be computed since the remaining ones
follow from symmetry arguments and the observation that 
$ \sum_{j} (\hat s_{x_i}^{T})_j =0$.

\begin{figure}[htb]
  %
%
%
%
%
%
  \begin{minipage}[h]{0.31\textwidth}
	\centering
	
	\begin{tikzpicture}[x={(0:1cm)}, y={(90:1cm)}, z={(210:0.5cm)}]
	\tikzset{font={\fontsize{8pt}{12}\selectfont}}
	
	\begin{scope}[auto, shift={(0,0)}, scale=2,every node/.style={draw,circle,minimum size=1.6em,inner sep=1},node distance=2cm]
	
	\node[draw,circle] at (-1,0,0) (mw) {$1'$};
	\node[draw,circle] at (0,0,0) (mc) {$0$};
	\node[draw,circle] at (1,0,0) (me) {$1$};
	\node[draw,circle] at (-1,0,-1) (mnw) {$3$};
	\node[draw,circle] at (0,0,-1) (mn) {$2$};
	\node[draw,circle] at (0,0,1) (ms) {$2'$};
	\node[draw,circle] at (1,0,1) (mse) {$3'$};
	
	\node[draw,circle] at (0,1,1) (ts) {$6$};
	\node[draw,circle] at (1,1,1) (tse) {$7$};
	\node[draw,circle] at (-1,1,0) (tw) {$5$};
	\node[draw,circle] at (0,1,0) (tc) {$4$};
	
	\node[draw,circle] at (0,-1,0) (bc) {$4'$};
	\node[draw,circle] at (1,-1,0) (be) {$5'$};
	\node[draw,circle] at (-1,-1,-1) (bnw) {$7'$};
	\node[draw,circle] at (0,-1,-1) (bn) {$6'$};

	\draw[->,line width=2pt] (mc) -- (me);
	\draw[->,line width=2pt] (mc) -- (mnw);
	\draw[->,line width=2pt] (mc) -- (mn);
	\draw[->,line width=2pt] (mc) -- (ts);
	\draw[->,line width=2pt] (mc) -- (tse);
	\draw[->,line width=2pt] (mc) -- (tw);
	\draw[->,line width=2pt] (mc) -- (tc);
	\draw[->,line width=2pt] (mc) -- (bc);
	\draw[->,line width=2pt] (mc) -- (be);
	\draw[->,line width=2pt] (mc) -- (bnw);
	\draw[->,line width=2pt] (mc) -- (bn);
	\draw[->,line width=2pt] (mc) -- (ms);
	\draw[->,line width=2pt] (mc) -- (mse);
	\draw[->,line width=2pt] (mc) -- (mw);
	
	\draw (mnw) -- (mw) -- (ms) -- (mse) -- (me) -- (mn) -- (mnw);
	\draw (tc) -- (tw) -- (ts) -- (tse) -- (tc);
	\draw (bc) -- (be) -- (bn) -- (bnw) -- (bc);
	
	\end{scope}
	
	\end{tikzpicture}
  \end{minipage}
  \begin{minipage}[h]{0.31\textwidth}
    \centering
    \includegraphics[width=0.85\textwidth]{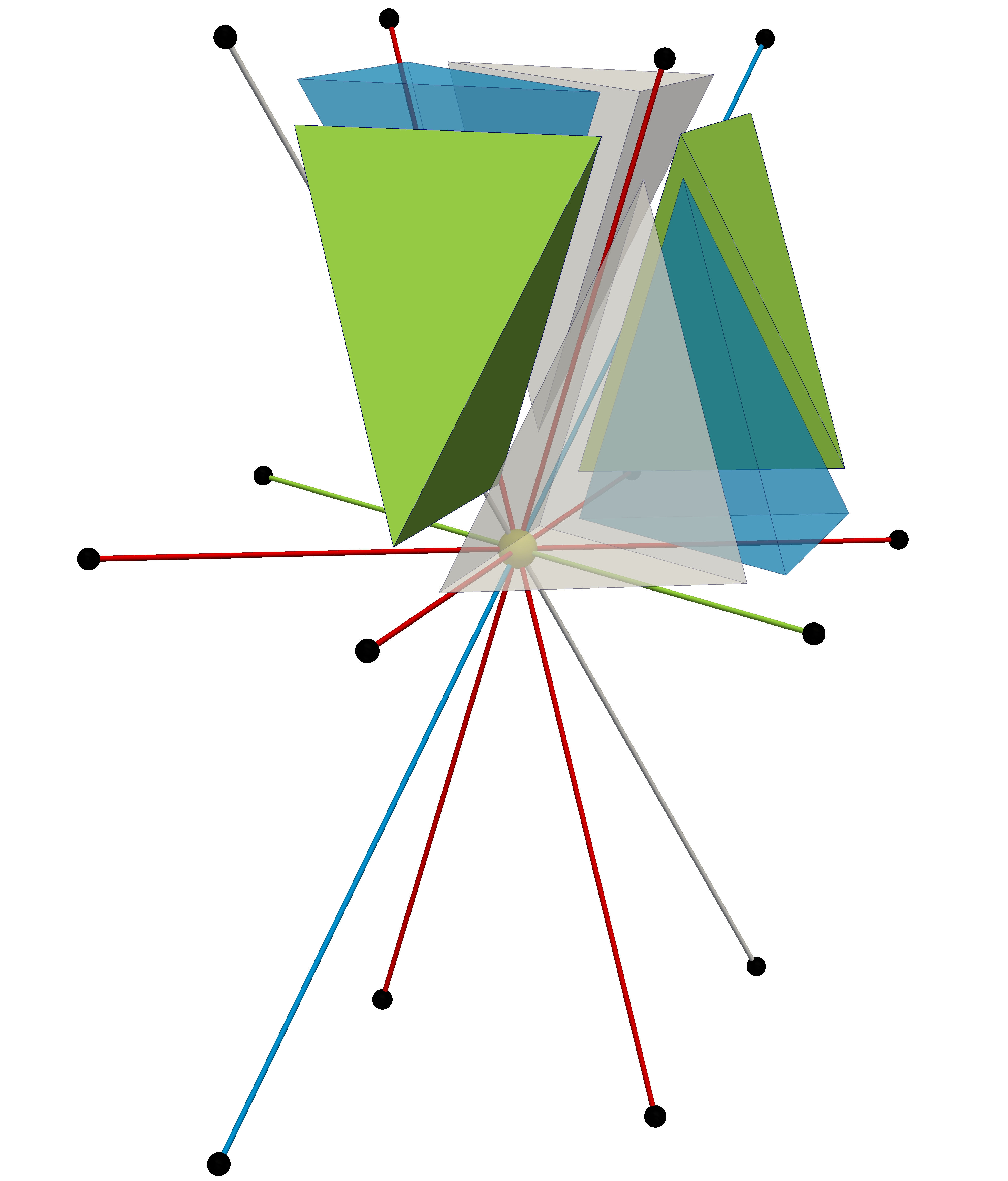}
 \end{minipage}
  \begin{minipage}[h]{0.31\textwidth}
    \centering
    \includegraphics[width=0.85\textwidth]{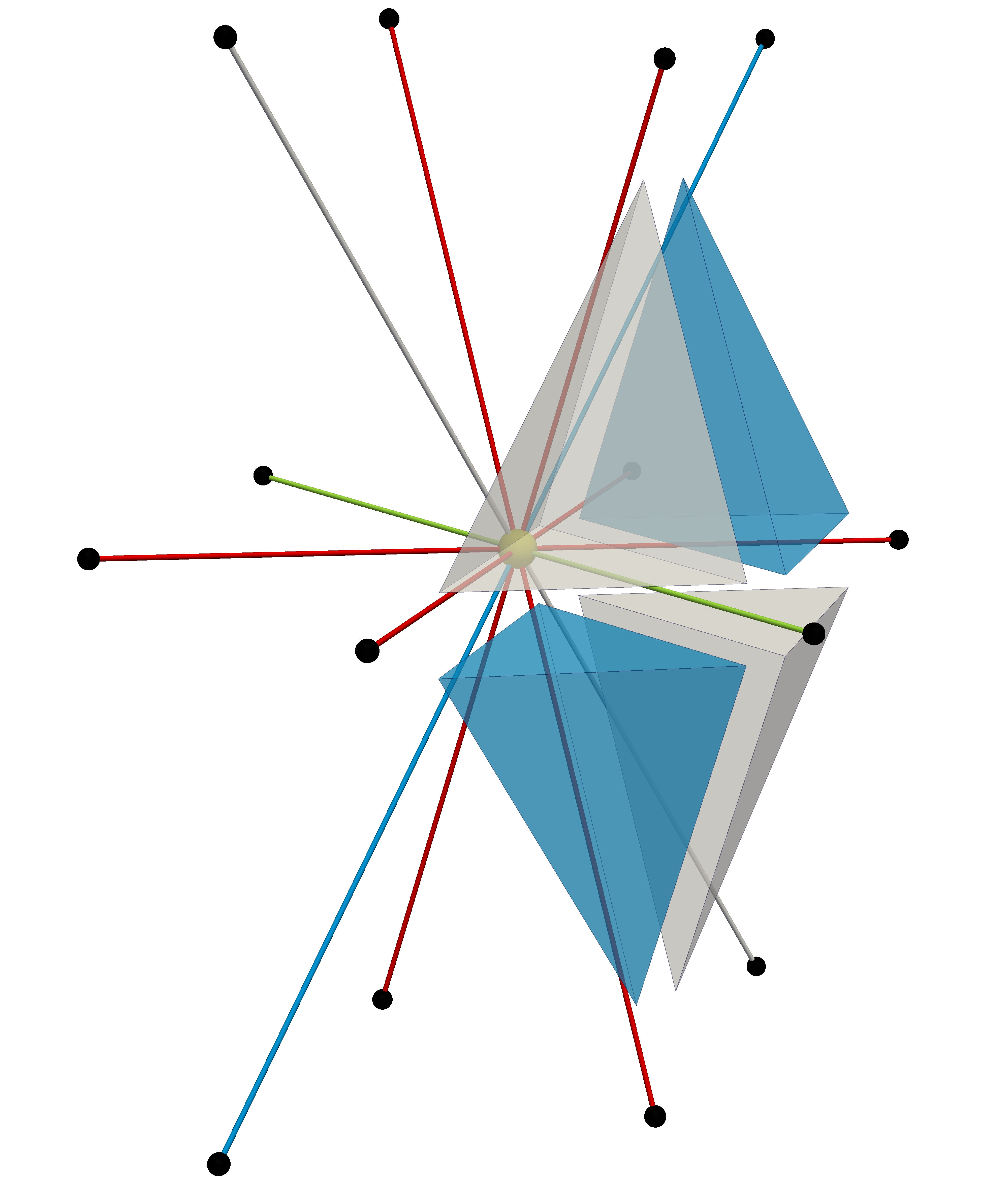}
    \end{minipage}
\caption{\label{fig:stencildir}
  From left to right:  Exemplary local indices $j$ and their corresponding direction vectors of a 15 point stencil in 3D; Six elements attached to one edge;  Four elements attached to one edge}
\end{figure}
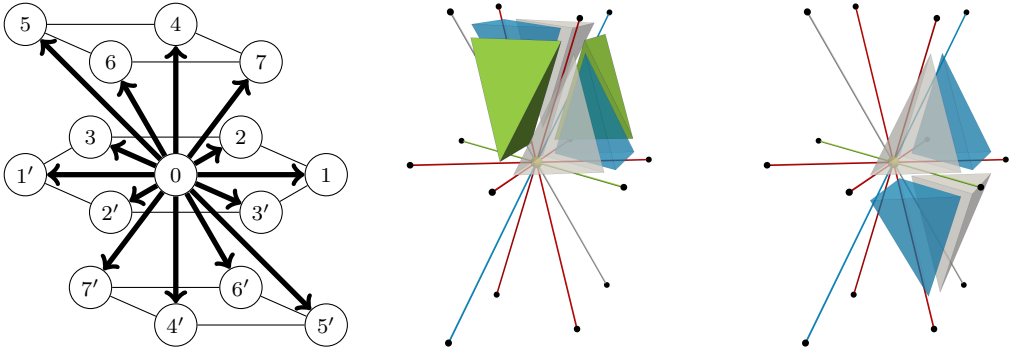

Due to the hierarchical hybrid grid structure, two stencils $\hat s_{x_{i_1}}^{T} $ and $\hat s_{x_{i_2}}^{T} $ are exactly the same if $x_{i_1}$ and $x_{i_2}$ are two nodes belonging to the same primitive; i.e., we find only $15$ different
kinds of stencils per macro-element in
3D\changed{, one for each of its 15 primitives (4 vertices, 6 edges, 4 faces, 1 volume),} and $ 7$ in 2D.
This observation allows for an extremely fast and memory-efficient on-the-fly (re)assembly of the entries of the stiffness matrix in stencil form.
For each node $x_i$ in the data structure, we save the nodal values of the  coefficient function $k$. With these considerations in mind, the bilinear form~\eqref{eq:discrete} can be evaluated
very efficiently and requires only a suitable scaling of the reference entries;
see Sec.~\ref{sec:cost} for detailed cost considerations.
%

\section{Variational crime framework and a priori analysis}
\label{sec:analysis}
 In order to obtain order $h$ and $h^2$ a priori estimates of the modified finite element approximation in the $H^1$- and
 $L^2$-norm, respectively, we analyze
 the discrete bilinear form.
 From now on, we assume that $k |_T \in W^{2, \infty} (T)$ for each $T \in {\mathcal T}_H$. Moreover, we denote by $\| \cdot \|_0 $ the $L^2$-norm on $\Omega$ and $\| \cdot \|_{{\infty}} := \sup_{T \in {\mathcal T}_{H} 
} \| \cdot \|_{{\infty}; T}$ defines  a broken $L^\infty$-norm. We recall that
the coefficient function is only assumed to be element-wise smooth with respect to the macro triangulation.
Existence and uniqueness of a finite element solution $u_h \in V_h \cap V_0$
of
\begin{align*} 
a_h (u_h, v_h) = f(v_h), \quad v_h \in V_h \cap V_0
\end{align*}
 is given provided that the following assumption (A1) holds true:
\begin{itemize}
\item[(A1)] $a_h( \cdot,\cdot) $ is uniformly coercive on $V_h \cap V_0$ 
\item[(A2)] $ | a(v_h,w_h ) - a_h(v_h,w_h) | \lesssim h  \| \nabla k \|_{{\infty}}  \| \nabla v_h \|_0 \| \nabla w_h \|_0 $, $\quad v_h,w_h \in V_h$
\end{itemize}
Here and in the following, the notation $\lesssim$ is used as abbreviation for
$\leq C$, where $C < \infty$ is independent of the mesh-size $h$. The assumption (A2), if combined with Strang's first lemma, yields that the finite element solution results in  ${\mathcal O} (h)$ a priori estimates with respect to the $H^1$-norm; see, e.g., \cite{CIA76,SF08}. We note that for $h$ small enough, the  uniform coercivity (A1) follows from the consistency assumption (A2), since
for $v_h \in V_h $
\begin{align*}
 a_h(v_h,v_h ) \geq a(v_h,v_h ) - |  a_h(v_h,v_h) - a(v_h,v_h)  | \geq C (1 - c h)
  \| \nabla v_h \|_0^2 .
\end{align*}

\begin{remark}
As it is commonly done in the finite element analysis in unweighted Sobolev norms, we allow the generic constant $C$ to be dependent on \changed{the global contrast of k defined by} $\sup_\Omega k / \inf_\Omega k$. \changed{Numerical results, however, show that the resulting bounds may be overly pessimistic for coefficients with large global variations. In \cite{PechsteinScheichl2013} and the references therein, methods to improve the bounds in this case are presented. The examples show that the bounds may be improved significantly for coefficients with a global contrast in the magnitude of about $10^5$.} \changed{We are mainly interested in showing alternative assembly techniques to the standard finite element method and in comparing them to the well-established approaches in standard norms. Moreover, in our modification only the local variation of the coefficient $k$ is important, therefore we shall not work out these subtleties here.}
\end{remark}

\subsection{Abstract framework for  \texorpdfstring{$\mathbf{L^2}$}{L2}-norm estimates}
Since (A2) does not automatically guarantee optimal order $L^2$-estimates, we employ duality arguments. To get a better feeling on the required accuracy of
$a_h(\cdot,\cdot)$, we briefly recall the basic steps occurring in the
proof of the upper bound. As it is standard, we assume
$H^2$-regularity of the primal and the dual problem. Restricting
ourselves to the case of homogeneous Dirichlet boundaries, the dual PDE
and boundary operators coincide with the primal ones. Let us denote by $P_h u$ the standard Galerkin approximation of $u$, i.e., the finite element solution obtained as the solution of a discrete problem using the bilinear form $a(\cdot,\cdot)$. It is well-known that under the given assumptions $\| u - P_h u \|_0 =
{\mathcal O} (h^2)$. Now, to obtain an $L^2$-estimate for $u_h$, we consider the dual problem with $u_h - P_h u$ on the right-hand side. Let $w \in V_0 $ be the solution of
$a(v,w) = (u_h - P_hu,v)_0$ for $v \in V_0$. Due to the standard
Galerkin orthogonality, we obtain
\begin{align} \label{eq:dual}
\| u_h - P_h u \|_0^2 = a(u_h -P_h u,w) = a(u_h - P_h u, P_h w) =
a(u_h, P_h w) - a_h(u_h, P_h w) .
\end{align}
This straightforward consideration shows us that compared to (A2), we need to make stronger assumptions on the mesh-dependent
bilinear form $a_h(\cdot,\cdot)$. We define (A3) by 
\begin{itemize}
 \item[(A3)] $
 | a(v_h,w_h ) - a_h(v_h,w_h) | \lesssim h^2 \| H k \|_{{\infty}}
 \| \nabla v_h \|_0 \| \nabla w_h \|_0 +  h \| \nabla k
 \|_{{\infty}}  \| \nabla v_h \|_{0;S_h} \| \nabla w_h \|_{0; S_h},
 $
\end{itemize}
where $ H k$ denotes the Hessian  of $k$ and $S_h := \cup_{T \in 
{\mathcal T}_{H}} S_h(T)$ with
$S_h(T) := \{ t \in {\mathcal T}_h; \partial t \cap \partial T \neq
\emptyset \} $; see  Fig.~\ref{fig:streifen} for a 2D illustration.
The semi-norm $\| \cdot \|_{0;S_h}$ stands for the $L^2$-norm restricted to $S_h$.

\begin{figure}[ht]
\centering
\includegraphics[width=0.25\textwidth]{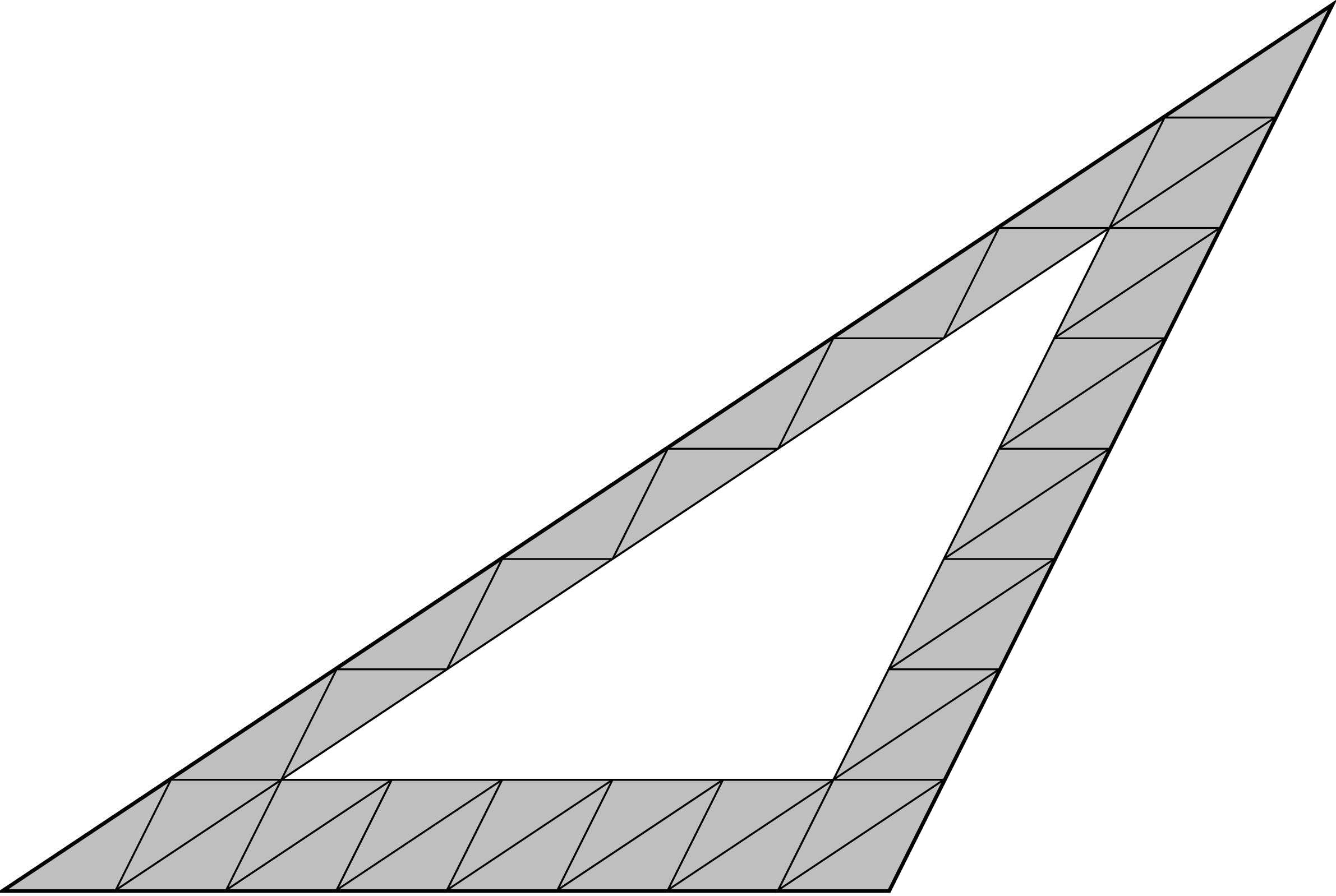}
\hspace*{2cm}
\includegraphics[width=0.4\textwidth]{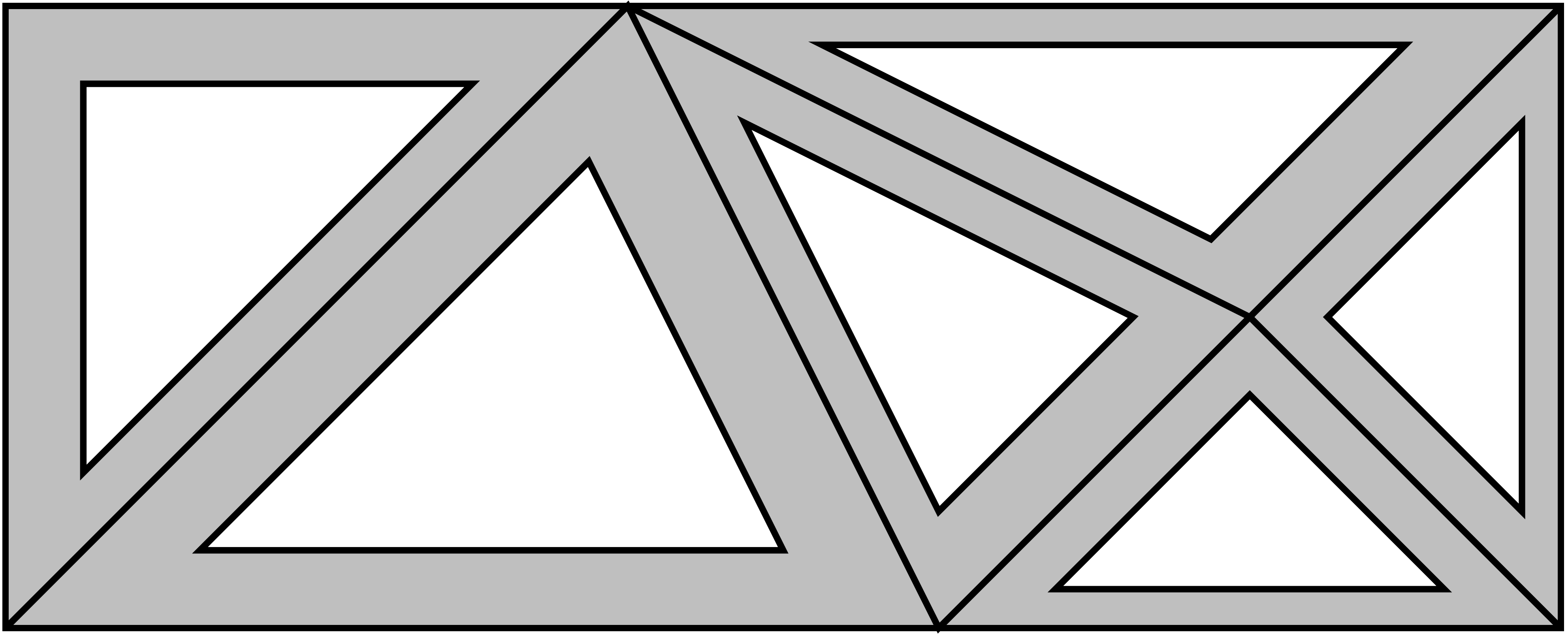}
\caption{Elements in $S_h (T) $ (left) and $S_h$ (right) for $d=2$ \label{fig:streifen}}
\end{figure}

\begin{lemma} \label{lem:optimalorder}
Let the problem under consideration be $H^2$-regular,  $h$ be sufficiently small and (A3) be satisfied.
%
Then we obtain a unique solution and optimal order convergence in the $H^1$- and the $L^2$-norm, i.e.,
\begin{align} \label{eq:optord}
\| u_h - u \|_0 + h \| \nabla(u_h - u) \|_0 \lesssim h^2 (\| H u \|_0 + \| Hk \|_{{\infty}} \| \nabla u \|_0 + \| \nabla k\|_{{\infty} } ( \| \nabla u \|_0 + \| Hu\|_0) ).
\end{align} 
\end{lemma}
\begin{proof}
Given that $h$ is small enough, (A1) follows from (A3). 
In terms of \eqref{eq:dual} and (A3), we get
\begin{align} \label{eq:zwischen}
\| u_h - P_h u \|_0^2 \lesssim h^2 \| H k \|_{{\infty} }  \| \nabla u_h \|_0  \| \nabla P_h w \|_0 + h \| \nabla k \|_{{\infty} }  \| \nabla u_h \|_{0;S_h} \|  \nabla P_h w \|_{0; S_h} .
\end{align}
The stability of the standard conforming Galerkin formulation yields
$  \| \nabla P_h w \|_0 \lesssim \| u_h - P_h u \|_0 $. By Definition
\eqref{eq:meshdep}, we find that  the discrete bilinear form $a_h (\cdot,\cdot) $ is uniformly continuous for a coefficient function in $L^\infty (\Omega)$,
and thus $   \| \nabla u_h \|_0 \lesssim \| \nabla u \|_0 $. 
To bound the two terms involving  $\| \nabla \cdot \|_{0;S_h}$, we use
the 1D Sobolev embedding $H^1( (0,1)) \subset L^\infty ((0,1))$ \cite{Tri95}. More precisely, for an element in $H^s(\Omega)$, $ s > 0.5$, we have
$
\| v \|_{0;S_h}\lesssim \| v \|_{H^s(\Omega)};
$
see \cite{LMWZ10}. Here we use $s=1$ and in terms of
the $H^2$-regularity assumption, we obtain
\begin{align*}
 \| \nabla u_h \|_{0;S_h}  &\leq \| \nabla(u -u_h)  \|_{0; S_h} + \| \nabla u \|_{0;S_h} \lesssim
 \| \nabla (u -u_h)  \|_{0; S_h} +  \sqrt{h} ( \| \nabla  u \|_0 + \| Hu \|_{0}) ,  \\
 \| \nabla P_h w \|_{0;S_h} & \leq \| \nabla(w -P_h w)  \|_{0; S_h} +   \sqrt{h} \| u_h -P_h u \|_{0}  \lesssim \sqrt{h} \| u_h -P_h u \|_{0} .
\end{align*}
Using the above estimates, \eqref{eq:zwischen} reduces to
\begin{align*}
\| u_h - P_h u \|_0 \lesssim  h^2 \left( \| H k \|_{{\infty} }  \| \nabla u\|_0   +  \| \nabla k \|_{{\infty} }  ( \| \nabla  u \|_0 + \| Hu \|_{0}) \right)  + h^{\frac 32} \| \nabla k \|_{{\infty} }  \| \nabla (u -u_h)  \|_{0}.
\end{align*}
 Applying the  triangle inequality and using the approximation properties of the Galerkin finite element solution results in the extra term $h^2 \| Hu \|_0$ in the upper bound for $\| u - u_h\|_0$.
The bound for $ h \| \nabla(u_h - u) \|_0  $ follows by a standard inverse estimate for elements in $V_h$ and the best approximation property of $V_h \cap V_0$. Since for $h$ small enough it holds $c_1 \sqrt{h} \| \nabla k\|_\infty \leq 1/2$ where $c_1 < \infty $ is a suitably fixed positive constant, the upper bound \eqref{eq:optord} follows.
\end{proof}

\subsection{Verification of the assumptions}
\label{sec:hhg}
It is well-known \cite{CIA76} that assumptions (A1)-(A3) are satisfied for the bilinear form 
\begin{align} \label{eq:standardfe}
\tilde a_h (u_h, v_h )  := \sum_{t \in {\mathcal T}_h} \frac{| t|}{d+1}
\sum_{i=1}^{d+1}  k|_t (x_i^t) \nabla u_h|_t (x_i^t) \cdot  \nabla v_h|_t (x_i^t)  = \sum_{t \in {\mathcal T}_h} \bar k_t
 \integrald{t}{}{\nabla u_h \cdot \nabla v_h}{x} ,
\end{align}
Here $x_i^t$ denotes the vertices of the d-dimensional simplex $t$, i.e.,
we approximate the integral by a nodal quadrature rule and $ \bar k_t :=(\sum_{i=1}^{d+1}  k|_t(x_i^t))/(d+1)$.
Thus, to verify the assumptions also for $a_h(\cdot,\cdot)$, it is sufficient to consider
$a_h(v_h,w_h) - \tilde a_h(v_h,w_h)$ in more detail with
$a_h(\cdot,\cdot) $ given by \eqref{eq:meshdep}.
Let
\begin{align*}
\small
\hat A_t := \begin{pmatrix} a_{1,1}^t & a_{1,2}^t&  \ldots & a_{1, d+1}^t \\
a_{1,2}^t &a_{2,2}^t& \ldots &a_{2 ,d+1}^t \\
\vdots & \vdots  & &  \vdots\\
a_{1 ,d+1 }^t & a_{2, d+1}^t & \ldots & a_{d+1, d+1}^t
\end{pmatrix}, \quad 
K_t := \begin{pmatrix} k_{1,1}^t & k_{1,2}^t&  \ldots & k_{1, d+1}^t \\
k_{1,2}^t & k_{2,2}^t& \ldots &k_{2, d+1}^t \\
\vdots & \vdots  & &  \vdots\\
k_{1, d+1 }^t & k_{2 ,d+1}^t & \ldots & k_{d+1, d+1}^t
\end{pmatrix}, 
\end{align*}
be the local stiffness matrix associated with the nodal basis functions
$\phi_i^t $, i.e.,
$a_{i,j}^t := \int_t \nabla \phi_i^t \cdot \nabla  \phi_j^t \ dx $ and the local coefficient function with $ k_{i,j}^t := \frac 12 ( k|_t(x_i^t) + k|_t(x_j^t))$, $ i \neq j$. The diagonal entries of $K_t$ are defined differently as
\begin{align} \label{eq:zerosum}
k_{i,i}^t := \frac{-1}{a_{i,i}^t} \sum_{j \neq i} k_{i,j}^t a_{i,j}^t .
\end{align}
 We introduce the component-wise Hadamard product between two matrices as
$ (B\circ C)_{ij} := B_{ij} C_{ij}$ and define
the rank one  matrix $\tilde K_t$ by $(\tilde K_t)_{ij} :=
\bar k_t$.

Due to the symmetry of $\hat A_t$ and the fact that the row sum of $\hat A_t$ is
equal to zero, we can rewrite the discrete bilinear forms.
With $v_h, w_h \in V_h$, we associate locally elements $\mathbf{v}_t,\mathbf{w}_t \in \mathbb{R}^{d+1}$
with $(\mathbf{v}_t)_i := v_h(x_i^t)$, $(\mathbf{w}_t)_i := w_h(x_i^t)$.
We recall that if $v_h, w_h \in V_h \cap V_0$ and $x_i^t$ is a boundary node, then $ v_h(x_i^t) =0 =  w_h(x_i^t)$.
\begin{lemma} \label{lem:equal}
The bilinear forms given by \eqref{eq:meshdep} and \eqref{eq:standardfe}
have the algebraic form
  \begin{align} \label{eq:equal}
 a_h (v_h, w_h )  &= \frac 12 \sum_{t \in {\mathcal T}_h}  \sum_{i,j=1}^{d+1}
 \left( (\mathbf{v}_t)_i -  (\mathbf{v}_t)_j\right) ( K_t \circ \hat A_t)_{ij}
 \left( (\mathbf{w}_t)_j  -  (\mathbf{w}_t)_i \right), \quad v_h, w_h \in V_h,
 \\ \label{eq:st}
\tilde  a_h (v_h, w_h )  &= \frac 12 \sum_{t \in {\mathcal T}_h}  \sum_{i,j=1}^{d+1}
 \left( (\mathbf{v}_t)_i -  (\mathbf{v}_t)_j\right) ( \tilde K_t \circ \hat A_t)_{ij}
 \left( (\mathbf{w}_t)_j  -  (\mathbf{w}_t)_i \right), \quad v_h, w_h \in V_h.
\end{align}
\end{lemma}
\begin{proof} 
We note that \eqref{eq:zerosum} yields that the row sum of $K_t
\circ \hat A_t$ is equal to zero. Moreover, $K_t \circ \hat A_t$ is by construction symmetric. Introducing
for $ T \in {\mathcal T}_H $ the set $ {\mathcal T}_h^{i,j;T}$ of all
elements $T \supset t \in {\mathcal T}_h$ sharing the global nodes $i$ and $j$,
 we identify by $i_t$ the local index of the node $i$ associated with the element $t$ and by $j_t$ the local index of the global node $j$. Then the standard local to global assembling process  yields that
the right-hand side in \eqref{eq:equal}  reads
 \begin{align*} 
\frac 12  \sum_{T \in {\mathcal T}_{H}}
 \sum_{i,j}  (\nu_{i} - \nu_j) (\chi_{j} - \chi_{i})    \sum_{t \in {\mathcal T}_h^{i,j;T}} k_{i_t j_t}^t a_{i_t j_t}^t .
\end{align*}
Comparing this result with the definition  \eqref{eq:discrete}, we find 
equality since the coefficient function is assumed to be smooth within each $T$.
The proof of \eqref{eq:st} follows with exactly the same arguments as the one for
\eqref{eq:equal}.
\end{proof}
Although the proof of the previous lemma is straightforward, the
implication of it for large scale simulations cannot be
underestimated. In 2D, the number of different edge types per macro-element is three  while in 3D it is seven
assuming uniform refinement.
All edges in 3D in the interior of a
macro-element $T$ share only four or six elements; see Fig. \ref{fig:stencildir}.
We have three edge types that have four elements attached to them and four edge types with six adjacent elements.

The algebraic formulations \eqref{eq:st} and \eqref{eq:equal}
allow us to estimate the effects of the 
variational crime introduced by the stencil scaling approach.
\begin{lemma} \label{lem:upper}
Assumptions (A2) and (A3) hold true.
\end{lemma}
\begin{proof}
The required  ${\mathcal O} (h) $ bound for (A2) is straightforward
and also holds true for any unstructured mesh refinement strategy. Recalling that $K_t =
\tilde K_t$ if the coefficient function $k$ restricted to $t$ is a constant, 
\eqref{eq:st} and \eqref{eq:equal} yield
\begin{align*}
| a_h (v_h, w_h ) - \tilde a_h (v_h,w_h) | &=  \left|
 \frac 12 \sum_{t \in {\mathcal T}_h}  \sum_{i,j=1}^{d+1}
 \left( (\mathbf{v}_t)_i -  (\mathbf{v}_t)_j\right) \left(
 ( K_t - \tilde K_t) \circ \hat A_t\right)_{ij}
 \left( (\mathbf{w}_t)_j  -  (\mathbf{w}_t)_i \right)\right|\\
& \! \! \! \lesssim h^{2-d}
\max_{t \in {\mathcal T}_h}  \left( \lambda_{\max} (K_t - \tilde K_t)
\lambda_{\max} (\hat A_t) \right)
\| \nabla v_h \|_0 \| \nabla w_h\|_0  \\
& \! \! \! \lesssim
\max_{t \in {\mathcal T}_h}  \left( \lambda_{\max} (K_t - \tilde K_t)
 \right)
\| \nabla v_h \|_0 \| \nabla w_h\|_0   \lesssim  h \| \nabla k
\|_{L^\infty}  \| 
\nabla v_h \|_0 \| \nabla w_h\|_0 ,
\end{align*}
where $\lambda_{\max}(\cdot) $ denotes the maximal eigenvalue of its
argument.

To show (A3), we have to exploit the structure of the mesh ${\mathcal T}_h$. Let ${\mathcal E}_h^T $ be the set of all edges in $\bar T$
and  ${\mathcal T}_h^{e;T}$ the subset of elements $t \subset T$
which share the edge $e$ having the two global nodes $i $ and $j$ as endpoints. As before, we identify the local indices of these endpoints by $i_t$ and $j_t$. We note that the two sets ${\mathcal
  E}_h^{T}$ and  ${\mathcal E}_h^{\hat T}$, $T \neq \hat T$  are not
necessarily disjoint. Observing that each element $t \subset T$ is
exactly contained in $\frac 12 d(d+1)$ elements of $
 {\mathcal T}_{h}^{e;T}$, we find
\begin{align*}
a_h (v_h, w_h ) -& \tilde a_h (v_h,w_h)  =\\
& \frac{1}{d(d+1)} \sum_{T \in {\mathcal T}_{H} }
\sum_{e \in {\mathcal E}_h^T}  \sum_{t \in {\mathcal T}_{h}^{e;T} }  ((\mathbf{v}_t)_{i_t} - (\mathbf{v}_t)_{j_t}) \left(( K_t - \tilde K_t) \circ \hat A_t\right)_{i_t j_t} ((\mathbf{w}_t)_{j_t} - (\mathbf{w}_t)_{i_t}), 
\end{align*}
and thus it is sufficient to focus on the contributions resulting from
$t \in {\mathcal T}_{h}^{e;T}$. We consider two cases separately: 
First, we consider the case that the edge $e$ is part of $\partial T$ for
at least one $T$, then we directly find the upper bound
 \begin{align*}
  h \| \nabla k \|_{L^\infty(t )} \| \nabla v_h\|_{t} \|\nabla w_h\|_t .
\end{align*}

Second, we consider the case  that $e$ is in the interior of one $T$. Then for each element $t  \in {\mathcal T}_{h}^{e;T}
$ there exists exactly one $t^m \in {\mathcal T}_{h}^{e;T}$ such that $t^m$ is obtained by point reflection at the midpoint of the edge $e$; see Fig.
\ref{fig:reflected}. In the following, we exploit that the midpoint of the edge is the barycenter of $t \cup t^m$. Here the local indices $i_t $ and $j_t$ are associated with the  global nodes $i$ and $j$, respectively. Without loss of generality, we assume a local renumbering such that $i_t = 1$, $j_t = 2$ and that
$ x_{i_t^m}^{t^m}$ is the point reflected vertex of $x_{i_t}^t$; see also Fig.~\ref{fig:reflected}.

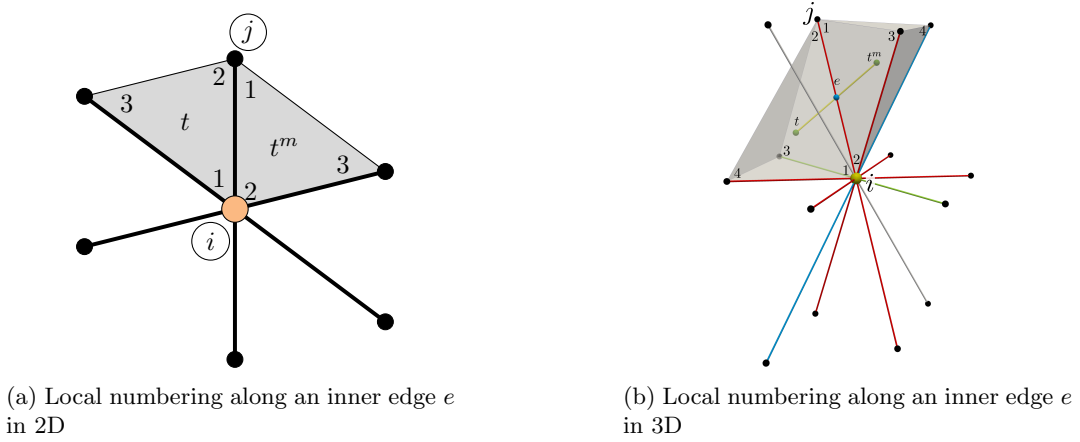
\begin{figure}[ht]
\begin{center}
\subcaptionbox{\label{fig:localnumber}Local numbering along an inner edge $e$ in 2D}[0.4\textwidth]{
\centering
\begin{tikzpicture}

\coordinate (one) at (0,0);
\coordinate (two) at (0,2);
\coordinate (three1) at (-2,1.5);
\coordinate (three2) at (2,0.5);
\coordinate (twoprime) at (0,-2);
\coordinate (three1prime) at (2,-1.5);
\coordinate (three2prime) at (-2,-0.5);

\draw[fill=lightgray!60] (one) -- (three1) -- (two);
\draw[fill=lightgray!60] (one) -- (three2) -- (two);
\draw[line width=1.5pt] (one) -- (two);
\draw[line width=1.5pt] (one) -- (three1);
\draw[line width=1.5pt] (one) -- (three2);

\draw[line width=1.5pt] (one) -- (twoprime);
\draw[line width=1.5pt] (one) -- (three1prime);
\draw[line width=1.5pt] (one) -- (three2prime);

\node[above left=5pt and 0pt] at (one) {1};
\node[below left] at (two) {2};
\node[below right=-4pt and 10pt] at (three1) {3};

\node[above right] at (one) {2};
\node[below right=5pt and 0pt] at (two) {1};
\node[above left=-4pt and 10pt] at (three2) {3};

\node[] at ($0.33*(one)+0.33*(two)+0.33*(three1)$) {$t$};
\node[] at ($0.33*(one)+0.33*(two)+0.33*(three2)$) {$t^m$};

\draw[black,fill=Apricot] (one) circle (5pt);
\draw[black,fill=black] (two) circle (3pt);
\draw[black,fill=black] (twoprime) circle (3pt);
\draw[black,fill=black] (three1) circle (3pt);
\draw[black,fill=black] (three2) circle (3pt);
\draw[black,fill=black] (three1prime) circle (3pt);
\draw[black,fill=black] (three2prime) circle (3pt);

\node[circle,below left=7pt and 4pt,draw,inner sep=2.5pt] at (one) {$i$};
\node[circle,above right=5pt and 0pt,draw,inner sep=1.5pt] at (two) {$j$};

\end{tikzpicture}}
\hspace*{2cm}
\subcaptionbox{\label{fig:center}Local numbering along an inner edge $e$ in 3D}[0.4\textwidth]{
\centering
\includegraphics[width=0.25\textwidth]{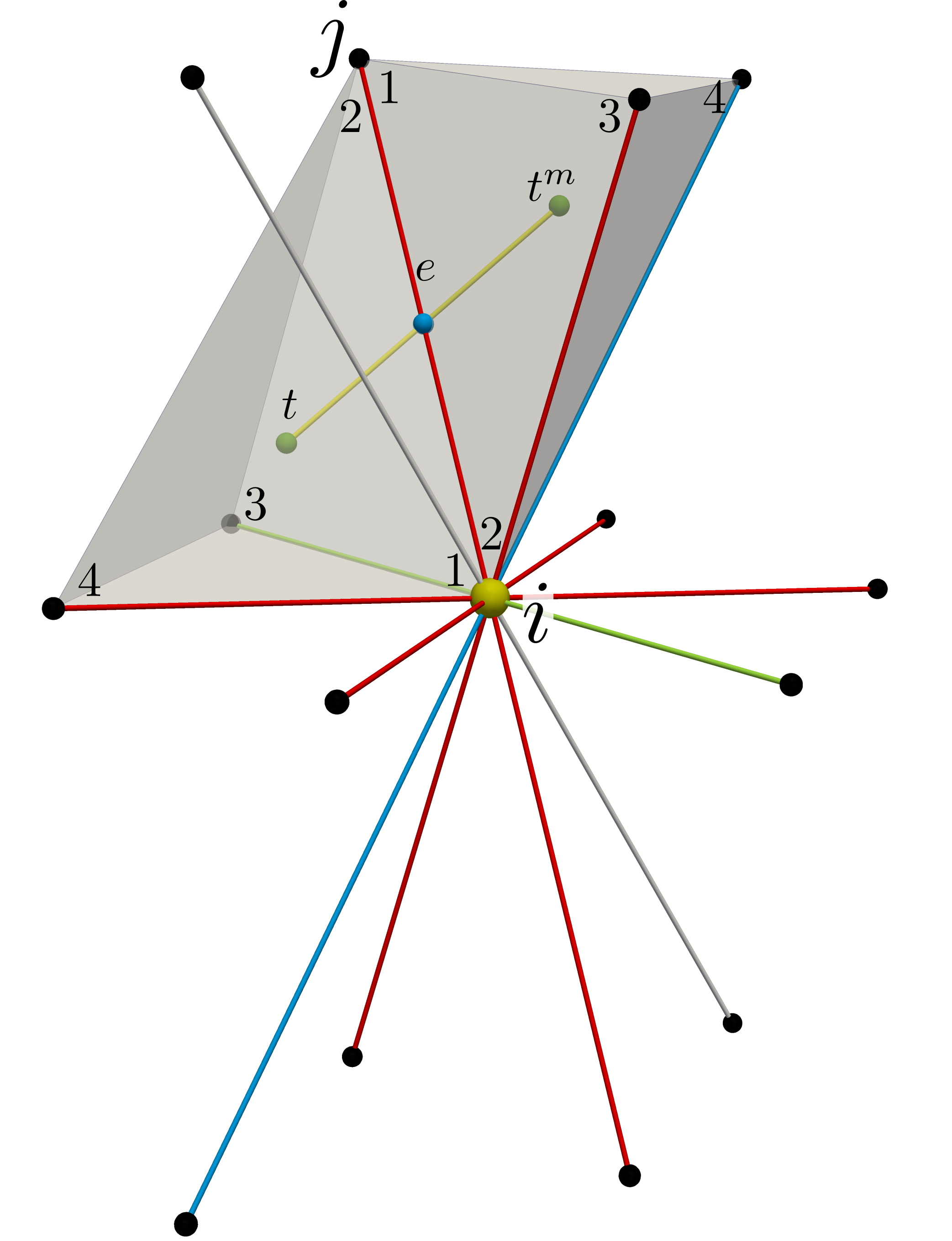}}
\end{center}
\caption{\label{fig:reflected}Local numbering in element $t$ and its point reflected element $t^m$ in 2D and 3D}
\end{figure}

Let us next focus on
\begin{align*}
 E_{ij}^t (v_h,w_h) :=& \left| ((\mathbf{v}_t)_{1} - (\mathbf{v}_t)_{2})  (( K_t - \tilde K_t) \circ \hat A_t )_{1,2} ((\mathbf{w}_t)_{2}  -  (\mathbf{w}_{t})_{1}) \right. \\ &\left. +
 ((\mathbf{v}_{t^m})_{2} - (\mathbf{v}_{t^m})_{1})  (( K_{t^m} - \tilde K_{t^m}) \circ A_{t^m} )_{2,1} ((\mathbf{w}_{t^m})_{1}  -  (\mathbf{w}_{t^m})_{2}) \right| .
\end{align*}
Exploiting the fact that $ (\hat A_t )_{1,2} = (\hat A_{t^m} )_{2,1}$, we can bound $E^t_{ij} (\cdot, \cdot)$ by the local $H^1$-seminorms
\begin{align*}
 E_{ij}^t (v_h,w_h) \lesssim &  \| \nabla v_h \|_{0;t \cup t^m}  \| \nabla w_h \|_{0;t \cup t^m} \left| (k|_T(x_i) + k|_T(x_j)) - \frac{1}{d+1} \sum_{l=1}^{d+1} (k|_T(x_l^t) + k|_T(x_l^{t^m})) \right| .  
\end{align*}
A Taylor expansion of $k$
in $\frac 12 (x^t_1 + x^{t^m}_1) = \frac 12 (x_i + x_j)$ guarantees that the
terms of zeroth and first order cancel out and only second order derivatives of $k$ scaled with $h^2$ remain, i.e.,
\begin{align*}
 E_{ij}^t (v_h,w_h) \lesssim &  h^2 \| H k \|_{L^\infty(t \cup t^m)} \| \nabla v_h \|_{0;t \cup t^m}  \| \nabla w_h \|_{0;t \cup t^m}  .  
\end{align*}
Then the summation over all macro elements, all edges, and all elements in the subsets ${\mathcal T}_l^{e;T}$ in combination with a finite covering argument~\cite{gariepy1992measure}
yields the upper bound of (A3).
\end{proof}

\section{Guaranteed uniform coercivity}
While for our hierarchical hybrid mesh framework the assumptions (A2) and (A3) are satisfied and thus asymptotically, i.e., for $h$ sufficiently small, also (A1) is satisfied, (A1) is not necessarily guaranteed for any given mesh $\mathcal{T}_h$.

\begin{lemma} \label{lem:pos}
If the matrix representation of the discrete Laplace operator is an M-matrix, then the scaled bilinear form $a_h(\cdot,\cdot) $ is  positive semi-definite on $V_h \times V_h$ for $k$ globally smooth. 
\end{lemma}
\begin{proof}
Let $\hat A_h $ be the matrix representation of the discrete Laplace operator, then the M-matrix property guarantees that $(\hat A_h)_{ij} \leq 0 $ for $ i \not=j$. Taking into account that $\sum_{T \in {\mathcal T}_H} \hat s_{ij}^T = (\hat A_h)_{ij} $, definition \eqref{eq:meshdep} for the special case $M=1$ and $k$ globally smooth yields
\begin{align*}
a_h(v_h , v_h) =\frac {-1}{4} 
 \sum_{i,j} (k  (x_{i}) +  k (x_{j}) )  (\nu_{i} - \nu_{j} )^2  (\hat A_h)_{ij} \geq 0.
\end{align*}
\end{proof}

\begin{remark}
  In 2D it is well-known \cite{ciarlet1973maximum} that if all elements of the macro mesh have no obtuse angle, then $\hat A_h$ is an M-matrix and we are in the setting of Lemma \ref{lem:pos}.
\end{remark}

\subsection{Pre-asymptotic modification in 2D based on (A2)}
\label{subsubsec:ma2}
Here we work out the technical details of a modification in 2D
that guarantees uniform ellipticity assuming that at least one macro-element $T$ has an obtuse
angle. 
Our  modification yields a linear condition on the local mesh-size depending on the discrete gradient of $k$.
It only applies to selected stencil directions.
In 2D our $7$-point stencil associated with an interior fine grid node has  exactly two positive off-center entries if the associated macro-element has an obtuse angle. We call the edges associated with a positive reference stencil entry
to be of gray type. 
With each macro-element $T$, we associate the reference stiffness matrix $\hat A_T$.
Without loss of generality, we assume that the local enumeration is done in such a way that the largest interior angle of the macro-element $T$ is located at the local node $3$, i.e., if $T$ has an obtuse angle then $a_{1,2}^T > 0$, $a_{1,3}^T < 0 $, and $a_{2,3}^T < 0$ and otherwise
$a_{i,j}^T \leq 0 $, $ 1 \leq i < j \leq 3$. By $\lambda_{\min}^T$ we denote the smallest  non-degenerated eigenvalue of the generalized eigenvalue problem
\begin{align*}
\hat A_T {\mathbf x} :=  \begin{pmatrix} a_{1,1}^T & a_{1,2}^T & a_{1, 3}^T \\
a_{1,2}^T &a_{2,2}^T &a_{2 ,3}^T \\
a_{1 ,3 }^T & a_{2, 3}^T & a_{3, 3}^T
\end{pmatrix} {\mathbf x} = \lambda
\left(\begin{array}{rrr} 2 & -1 & -1 \\
-1  & 2 & -1 \\
-1 & -1 & 2
\end{array}\right) {\mathbf x}, \quad {\mathbf x} \in {\mathbb R}^3 . 
\end{align*}
We note that both matrices in the eigenvalue problem are symmetric, positive semi-definite and have the same one dimensional kernel and thus $\lambda_{\min}^T > 0$.



Let $ e$ be a gray type edge. For each such edge $e$,  we possibly adapt our approach locally.  We denote by $\omega_{e;T}$ the element patch of all elements $t \in {\mathcal T}_h $, such that $t \subset T$ and $e \subset \partial t$. Then we define
 $$
 k_{e;\min} := \min_{\tilde e \in {{\mathcal E}_h^e}} k_{\tilde e}, \quad k_{\tilde e} :=
 \frac 12 (k|_T (x_1^{\tilde e}) + k|_T (x_2^{\tilde e}))
 $$
 where ${\mathcal E}_h^e$ is the set of all edges being in $\bar \omega_{e;T}$, and
 $x_1^{\tilde e} $ and $x_2^{\tilde e}$ are the two endpoints of $\tilde e$.
In the pre-asymptotic regime, i.e., if 
\begin{equation*}
( k_e - k_{e;\min} )  a_{1,2}^T >  k_{e;\min} \lambda_{\min}^T
\tag{MA2}\label{eqn:MA2}
\end{equation*}
we replace the scaling factor $k_e = \frac 12 (k|_T (x_1^e) + k|_T (x_2^e)) $ in
definition \eqref{eq:meshdep} by
\begin{align*}
  k_e^{\text{mod}} := k_{e;\min}  \left( 1 + \frac{\lambda_{\min}^T}{a_{1,2}^T}\right).
  \end{align*}
  Then it is obvious that $ k_{e;\min} <  k_e^{\text{mod}} < k_e$. We note that $2 a_{1,2}^T $ is the value of the 7-point stencil associated with a gray edge and thus trivial to access. 

\begin{lemma}
Let the bilinear form be modified according to \eqref{eqn:MA2}, then it is
uniformly positive definite on $V_h \cap V_0 \times V_h \cap V_0$ for all
simplicial hierarchical meshes.
\end{lemma}
\begin{proof} As it holds true for  standard bilinear forms also
our  modified one can be decomposed into element contributions.
 The local stiffness $A_t$ matrix for $t \subset T$ associated with
\eqref{eq:meshdep} reads $A_t= K_t \circ \hat A_t = K_t \circ \hat A_T$
 and can be rewritten in terms of $k_{\min}^t
:= \min_{1 \leq i < j \leq 3} k_{i,j}^t $ and $ \delta k_{i,j}^t := k_{i,j}^t - k_{\min}^t$ as
\begin{align} \small
A_t = k_{\min}^t \hat A_T +  \delta k_{1,2}^t
\begin{pmatrix} -a_{1,2}^T & a_{1,2}^T & 0\\
a_{1,2}^T &-a_{1,2}^T& 0 \\
0 & 0 & 0
\end{pmatrix}+  \delta k_{1,3}^t 
\begin{pmatrix} -a_{1,3}^T & 0 &  a_{1,3}^T \\
0 & 0& 0\\
a_{1,3}^T &0 & -a_{1,3}^T
\end{pmatrix}
+  \delta k_{2,3}^t 
\begin{pmatrix} 0 & 0 & 0\\
0 & -a_{2,3}^T &  a_{2,3}^T \\
0&a_{2,3}^T & -a_{2,3}^T
\end{pmatrix} \label{eq:decomp}
\end{align}
where we use a consistent local node enumeration.
We note that each $\delta k_{i,j}^t \geq 0$. 
Now we consider two cases separately.

First, let $t \subset T$ and $T$ be a macro-element having no obtuse angle then we find  that all four matrices on the right of \eqref{eq:decomp} are  positive semi-definite. Thus, we obtain
$A_t \geq k_{\min}^t \hat A_T $ and no modification is required.

Second, let $t \subset T$ and $T$ be a macro-element having one obtuse angle. Then the second  matrix on the right of  \eqref{eq:decomp} is negative semi-definite while the three other ones are positive semi-definite.
 Now we find
\begin{align*} 
{\mathbf x}^{\top} A_t  {\mathbf x} &\geq k_{\min}^t {\mathbf x}^{\top} \hat A_T {\mathbf x} + {\mathbf x}^{\top} \delta k_{1,2}^t
\begin{pmatrix} -a_{1,2}^T & a_{1,2}^T & 0\\
a_{1,2}^T &-a_{1,2}^T& 0 \\
0 & 0 & 0
\end{pmatrix}{\mathbf x} \\ &\geq \lambda_{\min}^T k_{\min}^t (  (x_1 -x_3)^2+ (x_2-x_3)^2)
+ ( \lambda_{\min}^T k_{\min}^t- a_{1,2}^T \delta k_{1,2}^t )(x_1 - x_2)^2 \\
 &\geq \lambda_{\min}^T k_{e;\min} (  (x_1 -x_3)^2+ (x_2-x_3)^2)
+ ( \lambda_{\min}^T k_{e;\min}- a_{1,2}^T \delta k_{1,2}^t )(x_1 - x_2)^2 ,
\end{align*}
where $e$ is the edge associated with the two local nodes 1 and 2.
Provided \eqref{eqn:MA2} is not satisfied, we have ${\mathbf x}^{\top} A_t  {\mathbf x}\geq \lambda_{\min}^T k_{e;\min} (  (x_1 -x_3)^2+ (x_2-x_3)^2) \geq
c k_{e;\min} {\mathbf x}^{\top} \hat A_T  {\mathbf x}$. If \eqref{eqn:MA2} is
satisfied, we do not work with the bilinear form \eqref{eq:meshdep}
but replace $k_{1,2}^t $ with $k_{e}^{\text{mod}}$. With this modification it is now obvious that the newly defined bilinear form is positive semi-definite on $V_h \times V_h$ and moreover positive definite on $V_h \cap V_0 \times V_h \cap V_0$. The coercivity constant depends only on the shape regularity of the macro-mesh, $\min k$, and the Poincar\'e--Friedrichs constant.
\end{proof}

\begin{remark}
From the proof it is obvious that any other positive scaling factor less or
equal to $k_e^{\text{mod}}$ also preserves the uniform ellipticity.
\end{remark}

\begin{remark}
We can replace in the modification criterion the local condition
\eqref{eqn:MA2} by
$$
l_T 2^{-(\ell+1)} \| \nabla k \|_{L^{\infty}(\omega_{e;T})} a_{1,2}^T \geq \inf_{x \in \bar \omega_{e;T}} k(x) \lambda_{\min}^T ,
$$
where ${\mathcal T}_h$ is obtained by $\ell $ uniform refinement steps from ${\mathcal T}_H$, and $l_T$ is the length of the second longest edge in $T$. Both these criteria allow a local marking of gray type edges which have to be modified.
To avoid computation of $k_{e;\min} $ each time it is needed, and thus further
reducing computational cost, we set the scaling factor for all gray type edges
in a marked $T$ to
\begin{equation}
\frac 12 (k_{\min} (x_1^e) + k_{\min} (x_2^e)), \quad 
\label{eqn:defKmin}
k_{\min} (x_i ) := \min\left\{k(x_j) \,|\, x_j \in \mathcal{N}_T(x_i)\right\}
\end{equation}
where $\mathcal{N}_T(x_i)$ denotes the set of all mesh nodes that are connected to node $x_i$
via an edge and belonging to the macro element $\bar{T}$ including $x_i$ itself. The quantity
$k_{\min}$ can be pre-computed for each node $x_i$ once at the beginning and stored as a node
based vector such as $k$ is. For non-linear problems where $k$ depends on the solution itself,
$k_{\min}$ can be updated directly after the update of $k$.
\end{remark}

The presented pre-asymptotic modification based on (A2) yields a condition on the local
mesh-size and only affects edge types associated with a positive stencil entry.
However, the proof of (A3) shows that a condition on the square of the mesh-size
is basically sufficient to guarantee (A1). This observation allows us to design
an alternative modification  yielding a condition  on the square of the local
mesh-size and involving the Hessian of $k$, except for the elements which are
in $S_h$. However, in contrast to the option discussed before, we possibly
also have to alter entries which are associated with negative reference stencil
entries, and therefore we do not discuss this case in detail. It is obvious that
for piecewise smooth $k$ there exists an $\ell_0$ such that for all refinement
levels $\ell \geq \ell_0$ no local modification has to be applied. This holds
true for both types of modifications. Thus, all the a priori estimates also hold true for our modified versions.
Since we are interested in piecewise moderate variations of $k$ and large scale
computations, i.e., large $\ell$, we assume that we are already in the
asymptotic regime, i.e., that no modification has to be applied for our 3D
numerical test cases, and we do not work out the technical details for the
modifications in 3D.

\subsection{Numerical counter example in the pre-asymptotic regime}
In the case that we are outside the setting of Lemma \ref{lem:pos},
it is easy to come up with an example where the scaled bilinear form is not
positive semi-definite, even for a globally smooth $k$. For a given mesh size $h$,
one can always construct a $k$ with a variation large enough such that the
scaled stiffness matrix has negative eigenvalues.

Here we consider a 2D setting on the unit square with an initial mesh which is not Delaunay; see left part of Fig.~\ref{fig:modelements}.
For the coefficient function $k$, we use a sigmoid function defined as
\begin{equation*}
k(x,y;m,\eta) = \frac{\eta}{1+\exp(-m(y-x-0.2))}+1,
\end{equation*}
with $m=50$ which yields a steep gradient. Further, we vary the magnitude of $k$ by setting $\eta \in \{1,10,100,1000\}$.

Now, we assemble the global stiffness matrix and report in Tab.~\ref{tab:eig} for the different choices of $\eta$ its minimal and maximal eigenvalue over a sequence
of uniform refinement steps. We show the eigenvalues for our scaling approach \eqref{eq:meshdep} with and without modification \eqref{eqn:MA2} and for the standard nodal integration assembly \eqref{eq:standardfe}.
\begin{table} \footnotesize
\caption{Minimal and maximal eigenvalues of the global stiffness matrix $A$. Italic entries highlight application of \eqref{eqn:MA2}.
 \label{tab:eig}} 
\begin{tabular}{r|cccccc|cccccc}
\hline
 \multicolumn{13}{c}{scaling approach} \\
\hline
  & \multicolumn{6}{c|}{$\lambda^A_\text{min}$} & \multicolumn{6}{c}{$\lambda^A_\text{max}$} \\
\hline
$\eta$ \textbackslash \hspace{0.2ex} $\ell$ &  0  &   1  &   2  &   3  &  4  &  5  &   0  &   1  &   2  &   3  &  4  &  5   \\
\hline
1       &     2.7   &    0.72  &   0.18  &   0.046   &    0.011  &    0.0029  &        9.1  &    13.3  &    15.1   &   17.4  &    19.8   &    21.3   \\
10      &     2.8   &    0.89  &   0.24  &   0.063   &    0.016  &    0.0039  &       40.0  &    69.3  &    80.7   &   93.2  &   108   &   116   \\
100     &    -8.9   &  -10.6     &  -5.5   &   0.066   &    0.018  &    0.0045  &        353  &    633  &   739      &  853  &   985   &  1066   \\
1000    &    -128   &   -148     & -94.5   &  -3.0     &    0.019  &    0.0049  &       3486  &    6265  &  7317     & 8450  &  9759   & 10562    \\
\hline
\multicolumn{13}{c}{scaling approach modified with \eqref{eqn:MA2}} \\
\hline
  & \multicolumn{6}{c|}{$\lambda^A_\text{min}$} & \multicolumn{6}{c}{$\lambda^A_\text{max}$} \\
\hline
$\eta$ \textbackslash \hspace{0.2ex} $\ell$ &  0  &   1  &   2  &   3  &  4  &  5  &   0  &   1  &   2  &   3  &  4  &  5   \\
\hline
1       &  2.7            & 0.72        & 0.18          & 0.046         & 0.011         &  0.0029  &    9.1 &   13.3   &   15.1  &    17.4  &    19.8  &     21.3  \\
10      &  \textit{3.8}   & \textit{1.1}& \textit{0.26} & 0.063         & 0.016         &  0.0039  &   \textit{40.8} &   \textit{69.4}   &   \textit{80.7}  &    93.2  &   108  &    116  \\
100     &  \textit{4.0}   & \textit{1.3}& \textit{0.30} & \textit{0.072}& \textit{0.018}&  0.0045  &  \textit{364}   &    \textit{634}   &  \textit{739}    &   \textit{853}    &   \textit{985}  &   1066  \\
1000    &  \textit{4.4}   & \textit{1.3}& \textit{0.33} & \textit{0.078}& \textit{0.019}&  0.0049  & \textit{3598}   &   \textit{6278}   & \textit{7321}    &  \textit{8453}     &  \textit{9759}  &  10562  \\
\hline
 \multicolumn{13}{c}{standard nodal integration} \\
\hline
  & \multicolumn{6}{c|}{$\lambda^A_\text{min}$} & \multicolumn{6}{c}{$\lambda^A_\text{max}$} \\
\hline
$\eta$ \textbackslash \hspace{0.2ex} $\ell$ &  0  &   1  &   2  &   3  &  4  &  5  &   0  &   1  &   2  &   3  &  4  &  5   \\
\hline
1       &    2.8  &       0.72  &     0.18  &   0.046  &   0.011  &   0.0029  &      9.0  &    12.8  &    14.7   &   17.1   &   19.7 &     21.2  \\
10      &    5.7  &       1.3   &     0.27  &   0.065  &   0.016  &   0.0040  &     36.6  &    64.9  &    78.4   &   90.7   &    107 &      116  \\
100     &   27.7  &       1.8   &     0.34  &   0.076  &   0.018  &   0.0045  &      316  &     590  &     717   &   829    &    977 &     1064  \\
1000    &    247  &       1.9   &     0.38  &   0.084  &   0.020  &   0.0049  &     3116  &    5840  &    7104   &  8211    &   9682 &    10549 \\
\hline
\end{tabular}
\end{table}
We find that for $\eta \geqslant 100$ the stiffness matrix of the scaling
approach has negative eigenvalues in the pre-asymptotic regime. In these cases
three or four refinement steps are required, respectively, to enter the asymptotic
regime. However, the positive definiteness of the matrix can be recovered on
coarser resolutions if \eqref{eqn:MA2} is applied. 
Asymptotically, the minimal or maximal eigenvalues of all three approaches tend to the same values.

In Fig.~\ref{fig:modelements}, we illustrate the action of the modification and show how the region of elements where it has to be applied is getting
smaller and finally vanishes with increasing number of refinement steps. 
To further illustrate how the region that is affected by \eqref{eqn:MA2} changes, we show a second example where all initial elements have an obtuse angle.
The underlying color bar represents the coefficient function $k$ with $\eta = 1000$ for the first (left), and $\eta=100$ for the second (right) example.
The location of the steepest gradient of $k$ is marked by a dashed line.
Note that the gradient of $k$ is constant along lines parallel to the dashed one.
For all edges that are modified by \eqref{eqn:MA2} the adjacent elements are shown. 
The color intensity refers to the refinement level and goes from bright (initial mesh) to dark.
\begin{figure}[ht]
\centering
\includegraphics[width=0.425\textwidth]{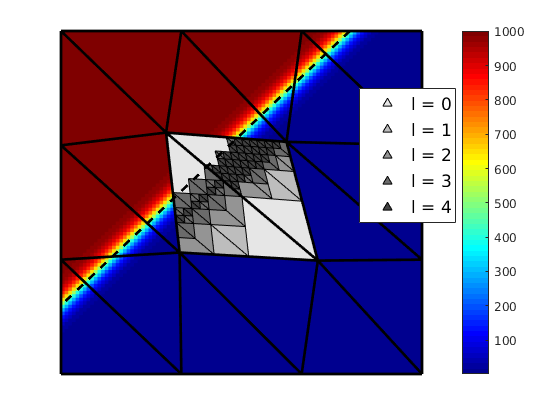}
\hspace*{1.5cm}
\includegraphics[width=0.425\textwidth]{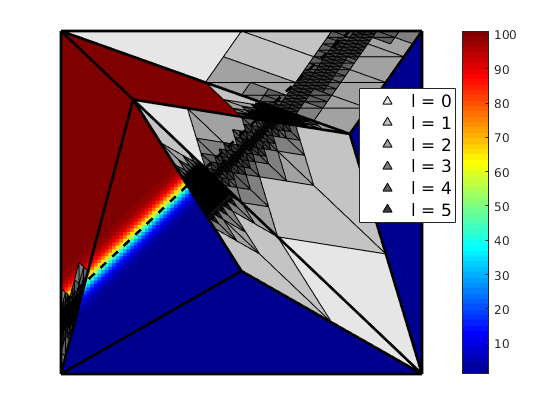}
\caption{Application of \eqref{eqn:MA2}:
  Mesh with two critical macro elements  used for Tab.~\ref{tab:eig} (left) and
  mesh where all macro elements have an obtuse angle (right).}
\label{fig:modelements}
\end{figure}

\subsection{The sign of the stencil entries in 3D}
In contrast to the 2D setting, a macro-mesh with no obtuse angle does not yield
that $\hat A_h$ is an M-matrix. 
Here the uniform refinement rule yields that for each macro element three
sub-classes of tetrahedra (gray, blue, green) exist. To each of these we
associate one interior edge type (gray, blue, green) defined by not
being parallel to any of the six edges of the respective tetrahedron type.
Fig.~\ref{fig:subtypes} shows the sub-classes and as example the gray edge type.
The coloring of the sub-classes is up to now arbitrary. We always associate the gray color with the macro-element and  call the associated interior edge, a gray type edge. The interior edges associated with the blue and green elements are called blue and green type edges, respectively. All other remaining edges are by notation red type edges.
\begin{figure}[ht]
\hfill
\includegraphics[width=0.25\textwidth]{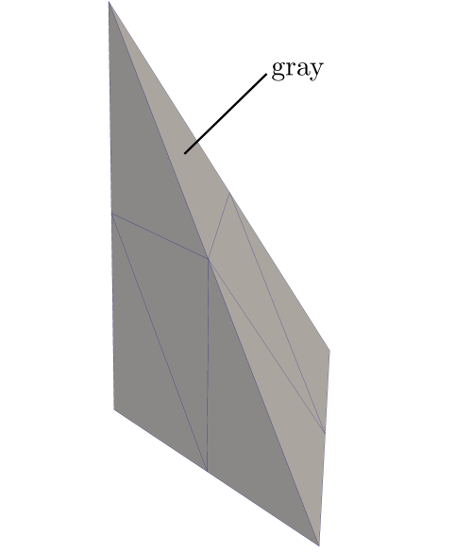}
\hfill
\includegraphics[width=0.25\textwidth]{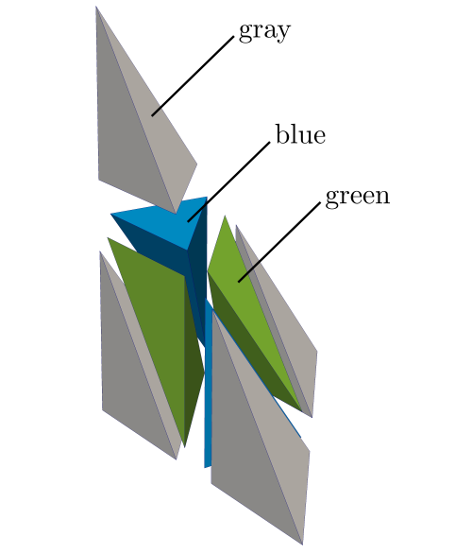}
\hfill
\includegraphics[width=0.32\textwidth]{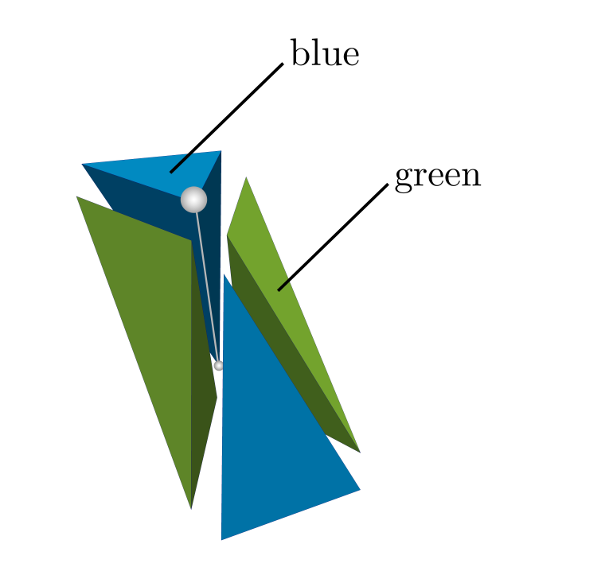}
\hfill
\caption{Uniform refinement of one macro-element $T$ (left) into three
subclasses (middle); gray edge between blue and green sub-tetrahedra (right).
}
\label{fig:subtypes}
\end{figure}
If the macro-element $T$ has no obtuse angle between two faces, then it follows from \cite{KQ95,KKN01} that  the reference stencil entries, i.e., the entries
associated with the Laplace operator,  associated with  gray type
edges  have a positive sign; see Fig.~\ref{fig:stencilsign}. However, it can be shown by some simple geometrical considerations that not both, green and blue type edges, can have a positive sign. If one of these has a positive sign, we call this the blue- and the other one the green-type edge. Conversely, if both have a negative sign, the coloring is arbitrary. In case the macro-element $T$ has no obtuse angle, the sign associated with all red type edges is automatically not positive.
Thus, we find for such elements in our 15-point stencil either two or four positive off-diagonal entries. Then a modification similar to the one proposed for the 2D case can be now applied to the edges of gray and possibly blue  type.
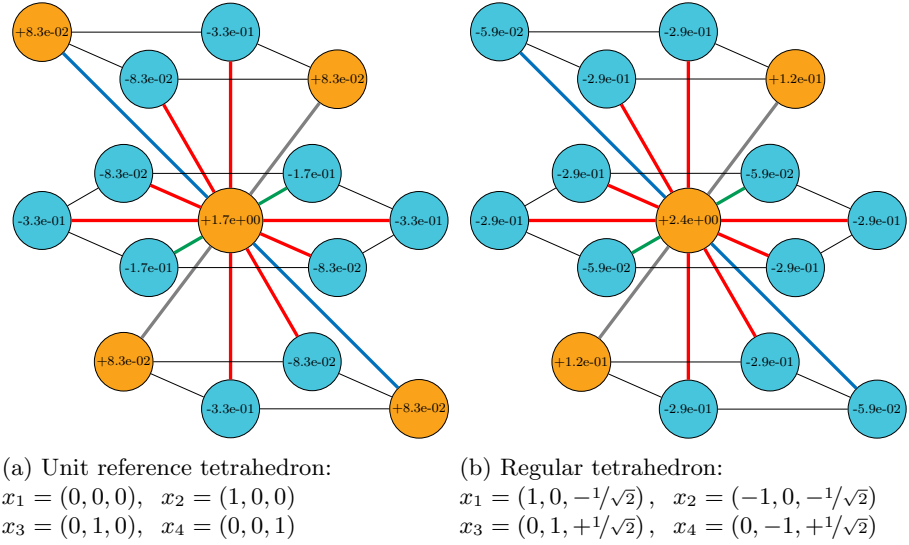
\begin{figure}[ht]
\begin{center}
	\subcaptionbox{
		Unit reference tetrahedron:\\
	$x_1 = (0,0,0),\;\;x_2 = (1,0,0)$\\
	$x_3 = (0,1,0),\;\;x_4 = (0,0,1)$}[0.4\textwidth]{
		\resizebox{5.8cm}{!}{\begin{tikzpicture}[x={(0:1cm)}, y={(90:1cm)}, z={(210:0.5cm)}]
		\tikzset{font={\fontsize{8pt}{12}\selectfont}}
		\begin{scope}[auto, shift={(0,0)}, scale=4,every node/.style={draw,circle,minimum size=3.5em,inner sep=1},node distance=2cm]
		\node[draw,circle,fill=SkyBlue] at (-1,0,0) (mw) {-3.3e-01};
		\node[draw,circle,fill=YellowOrange] at (0,0,0) (mc) {+1.7e+00};
		\node[draw,circle,fill=SkyBlue] at (1,0,0) (me) {-3.3e-01};
		\node[draw,circle,fill=SkyBlue] at (-1,0,-1) (mnw) {-8.3e-02};
		\node[draw,circle,fill=SkyBlue] at (0,0,-1) (mn) {-1.7e-01};
		\node[draw,circle,fill=SkyBlue] at (0,0,1) (ms) {-1.7e-01};
		\node[draw,circle,fill=SkyBlue] at (1,0,1) (mse) {-8.3e-02};
		\node[draw,circle,fill=SkyBlue] at (0,1,1) (ts) {-8.3e-02};
		\node[draw,circle,fill=YellowOrange] at (1,1,1) (tse) {+8.3e-02};
		\node[draw,circle,fill=YellowOrange] at (-1,1,0) (tw) {+8.3e-02};
		\node[draw,circle,fill=SkyBlue] at (0,1,0) (tc) {-3.3e-01};
		\node[draw,circle,fill=SkyBlue] at (0,-1,0) (bc) {-3.3e-01};
		\node[draw,circle,fill=YellowOrange] at (1,-1,0) (be) {+8.3e-02};
		\node[draw,circle,fill=YellowOrange] at (-1,-1,-1) (bnw) {+8.3e-02};
		\node[draw,circle,fill=SkyBlue] at (0,-1,-1) (bn) {-8.3e-02};
		\draw[line width=2pt, red] (mc) -- (me);
		\draw[line width=2pt, red] (mc) -- (mnw);
		\draw[line width=2pt, ForestGreen] (mc) -- (mn);
		\draw[line width=2pt, red] (mc) -- (ts);
		\draw[line width=2pt, gray] (mc) -- (tse);
		\draw[line width=2pt, RoyalBlue] (mc) -- (tw);
		\draw[line width=2pt, red] (mc) -- (tc);
		\draw[line width=2pt, red] (mc) -- (bc);
		\draw[line width=2pt, RoyalBlue] (mc) -- (be);
		\draw[line width=2pt, gray] (mc) -- (bnw);
		\draw[line width=2pt, red] (mc) -- (bn);
		\draw[line width=2pt, ForestGreen] (mc) -- (ms);
		\draw[line width=2pt, red] (mc) -- (mse);
		\draw[line width=2pt, red] (mc) -- (mw);
		\draw (mnw) -- (mw) -- (ms) -- (mse) -- (me) -- (mn) -- (mnw);
		\draw (tc) -- (tw) -- (ts) -- (tse) -- (tc);
		\draw (bc) -- (be) -- (bn) -- (bnw) -- (bc);
		\end{scope}
		\end{tikzpicture}}
	}
	\subcaptionbox{Regular tetrahedron:\\
	$x_1 = \left(1,0,-\nicefrac{1}{\sqrt{2}}\right),\;\;x_2 = \left(-1,0,-\nicefrac{1}{\sqrt{2}}\right)$\\
	$x_3 = \left(0,1,+\nicefrac{1}{\sqrt{2}}\right),\;\;x_4 = \left(0,-1,+\nicefrac{1}{\sqrt{2}}\right)$}[0.4\textwidth]{
		\resizebox{5.8cm}{!}{%
			\begin{tikzpicture}[x={(0:1cm)}, y={(90:1cm)}, z={(210:0.5cm)}]
			\tikzset{font={\fontsize{8pt}{12}\selectfont}}
			\begin{scope}[auto, shift={(0,0)}, scale=4,every node/.style={draw,circle,minimum size=3.5em,inner sep=1},node distance=2cm]
			\node[draw,circle,fill=SkyBlue] at (-1,0,0) (mw) {-2.9e-01};
			\node[draw,circle,fill=YellowOrange] at (0,0,0) (mc) {+2.4e+00};
			\node[draw,circle,fill=SkyBlue] at (1,0,0) (me) {-2.9e-01};
			\node[draw,circle,fill=SkyBlue] at (-1,0,-1) (mnw) {-2.9e-01};
			\node[draw,circle,fill=SkyBlue] at (0,0,-1) (mn) {-5.9e-02};
			\node[draw,circle,fill=SkyBlue] at (0,0,1) (ms) {-5.9e-02};
			\node[draw,circle,fill=SkyBlue] at (1,0,1) (mse) {-2.9e-01};
			\node[draw,circle,fill=SkyBlue] at (0,1,1) (ts) {-2.9e-01};
			\node[draw,circle,fill=YellowOrange] at (1,1,1) (tse) {+1.2e-01};
			\node[draw,circle,fill=SkyBlue] at (-1,1,0) (tw) {-5.9e-02};
			\node[draw,circle,fill=SkyBlue] at (0,1,0) (tc) {-2.9e-01};
			\node[draw,circle,fill=SkyBlue] at (0,-1,0) (bc) {-2.9e-01};
			\node[draw,circle,fill=SkyBlue] at (1,-1,0) (be) {-5.9e-02};
			\node[draw,circle,fill=YellowOrange] at (-1,-1,-1) (bnw) {+1.2e-01};
			\node[draw,circle,fill=SkyBlue] at (0,-1,-1) (bn) {-2.9e-01};
			\draw[line width=2pt, red] (mc) -- (me);
			\draw[line width=2pt, red] (mc) -- (mnw);
			\draw[line width=2pt, ForestGreen] (mc) -- (mn);
			\draw[line width=2pt, red] (mc) -- (ts);
			\draw[line width=2pt, gray] (mc) -- (tse);
			\draw[line width=2pt, RoyalBlue] (mc) -- (tw);
			\draw[line width=2pt, red] (mc) -- (tc);
			\draw[line width=2pt, red] (mc) -- (bc);
			\draw[line width=2pt, RoyalBlue] (mc) -- (be);
			\draw[line width=2pt, gray] (mc) -- (bnw);
			\draw[line width=2pt, red] (mc) -- (bn);
			\draw[line width=2pt, ForestGreen] (mc) -- (ms);
			\draw[line width=2pt, red] (mc) -- (mse);
			\draw[line width=2pt, red] (mc) -- (mw);
			\draw (mnw) -- (mw) -- (ms) -- (mse) -- (me) -- (mn) -- (mnw);
			\draw (tc) -- (tw) -- (ts) -- (tse) -- (tc);
			\draw (bc) -- (be) -- (bn) -- (bnw) -- (bc);
			\end{scope}
			\end{tikzpicture}}
	}
\end{center}
\caption{\label{fig:stencilsign}
Stencil entries colored by their sign at an inner node of two times refined tetrahedra, each without any obtuse angles between faces. On the left~(A), the stencil of the unit reference tetrahedron and on the right~(B), the stencil of a regular tetrahedron is depicted. The gray edge corresponds to the interior edge through $\frac 12 \left(x_1+x_3\right)$ and $\frac 12 \left(x_2+x_4\right)$. The green edge is the one in direction between $x_1 $ and $x_3$ and the blue one in direction  between $x_2$ and $x_4$. All other edge directions are marked in red.}
\end{figure}
The situation can be drastically different if the macro-element $T$ has an obtuse angle between two faces. In that case, we find up to four edge directions which carry a positive sign in the stencil.

\section{Reproduction property and primitive concept}
As already mentioned, 
our goal is to reduce the cost of a WU and the run-times while preserving discretization errors that are
qualitatively and quantitatively on the same level as for standard conforming
finite elements. Here we focus on the 3D case, but similar results can be
obtained for the 2D setting. The a priori bounds of Lemma
\ref{lem:optimalorder} do not necessarily guarantee that for an affine
coefficient function $k$ and affine solution $u$ the error is equal to zero.
In the upper bound \eqref{eq:optord} a term of the form
$\| \nabla k \|_{\infty} \| \nabla u \|_0 $ remains. A closer look at the proof reveals that this non-trivial contribution can be traced back to the terms associated with nodes  on the boundary of a macro-element.
This observation motivates us to introduce a modification of our stencil scaling approach. As already mentioned, all nodes are grouped into primitives and we have easy access to the elements of these primitives.
Recall that $\tilde a_h(\cdot,\cdot)$ defined as in \eqref{eq:standardfe} is associated with the standard finite element approach with nodal quadrature.
Let us by ${\mathcal W}_V$, ${\mathcal W}_E$, and ${\mathcal W}_F$ denote the set of all nodes associated with the vertex, edge, and face primitives, respectively. Now we introduce a modified stencil scaling approach, cf.~\eqref{eq:meshdep} and \eqref{eq:standardfe},
\begin{equation} \label{eq:modst}
  a_{h;{\mathcal I}} (v_h,\phi_i) := \left\{
  \begin{array}{ll}
    \tilde a_h(v_h , \phi_i), \quad & i \in {\mathcal I} \\
    a_h (v_h, \phi_i) & i \not\in {\mathcal I} 
  \end{array}  \right. .
  \end{equation}
Replacing  $a_h (\cdot,\cdot) $ by $a_{h;{\mathcal I}} (\cdot,\cdot)$ with
${\mathcal I} \subset {\mathcal W} :=  {\mathcal W}_V \cup  {\mathcal W}_E
\cup {\mathcal W}_F$ still yields that all node stencils associated with a node in a volume primitive are cheap to assemble. The number of node stencils which  have to be more expensively assembled grow only at most with $4^\ell$ while the total number of nodes grows with $8^\ell$.

\begin{remark}
The modification \eqref{eq:modst} introduces an asymmetry in the definition of
the stiffness matrix to which multigrid solvers 
are not sensitive. Moreover, the asymmetry tends asymptotically to zero with ${\mathcal O}(h^2) $.
\end{remark}

\begin{lemma} \label{lem:repro}
Let ${\mathcal I} =  {\mathcal W}$
or  ${\mathcal I} = {\mathcal W}_V \cup  {\mathcal W}_E $,
then an affine solution can be reproduced if $k$ is affine, i.e., $u_h = u$.
\end{lemma}
\begin{proof}
  For $i \not\in  {\mathcal W} $, associated with a node in the macro-element $T$
  and $k$ affine, we have
$
k_{ij} (\hat s_{i}^T)_j = (s_{i}^T)_j$,
and thus the bilinear form $a_{h; {\mathcal I}}  (\cdot,\cdot)$ is identical with $\tilde a_h(\cdot,\cdot)$ for ${\mathcal I} = {\mathcal W}$. Since then no variational crime occurs, we find $ u_h =u$ for any affine solution $u$.

The case ${\mathcal I}= {\mathcal W}_V \cup  {\mathcal W}_E$ is more involved.
Here the bilinear forms $a_{h; {\mathcal I}}  (\cdot,\cdot)$  and  $\tilde a_h(\cdot,\cdot)$ are not identical. However, it can be shown that for any affine function $v_{\text{aff}}$ we find
$$
a_{h; {\mathcal I}}  (v_{\text{aff}},\cdot) = \tilde a_h(v_{\text{aff}},\cdot)$$
on $V_h \cap V_0$. To see this, we follow similar arguments as in the proof of Lemma \ref{lem:upper}. This time we do not use a point reflected element $t^m$
but
a  shifted element $t^s$; see Fig. \ref{fig:shift}. We note that $\nabla v_{\text{aff}} $ is constant and
\begin{align} \label{eq:shift}
\frac 12 (k(x_i) + k(x_i + w_j)) - \frac 12 (k(x_i) +k(x_i + w_{j'})) =
k (x_c^t) - k(x_c^{t^s}) .
\end{align}
Here $x_c^t $ and $x_c^{t^s}$ stand for the barycenter of the elements $t$ and $t^s$, respectively. Two of the four vertices of element $t$ are the nodes $x_i$ and $x_i + w_j$. The shifted element $t^s$ is obtained from $t$ by a shift of $- w_j$.

\begin{figure}[h]
\subcaptionbox{Element $t$ and  shifted one $t^s$}[0.35\textwidth]{
\includegraphics[width=0.18\textwidth]{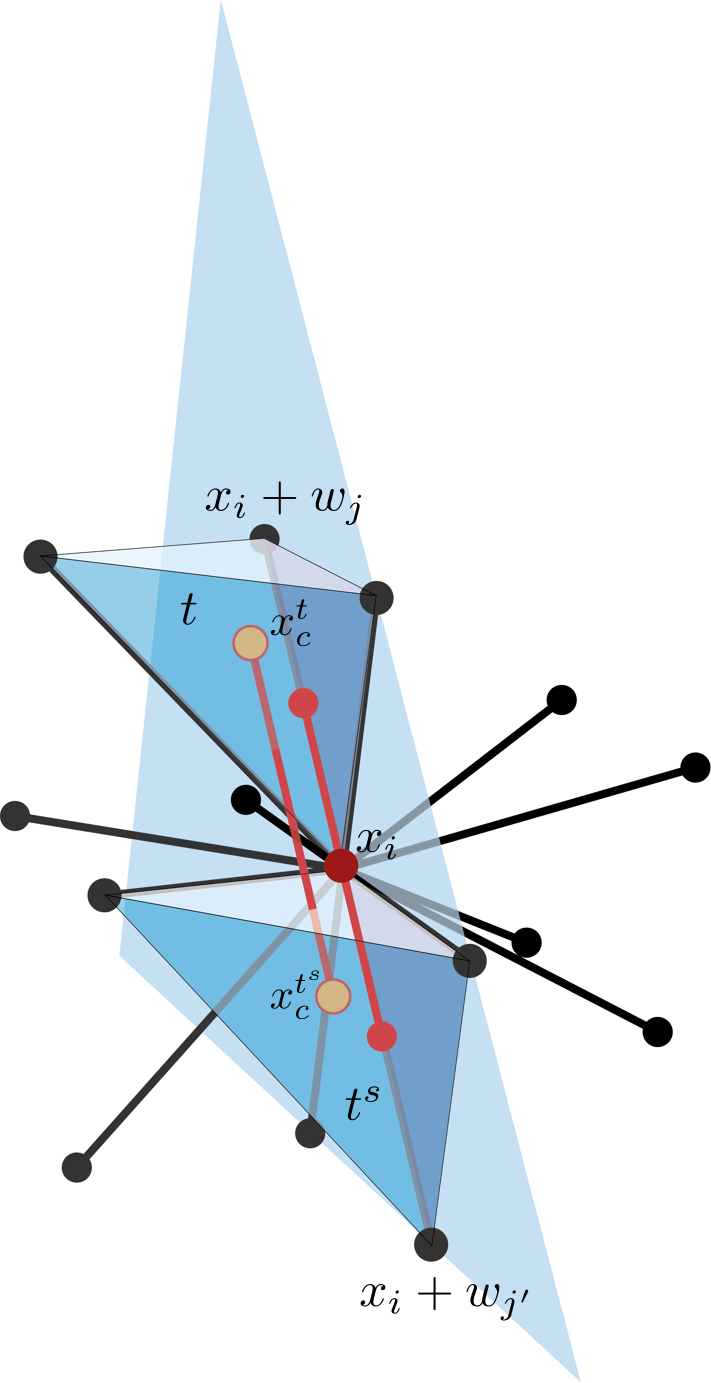}}
\subcaptionbox{Reproduction property}[0.625\textwidth]{
\begin{tabular}{l|c}
\hline
\mbox{ }\\[-0.25cm]
Used bilinear form  to    & Discretization error\\
define FE approximation & in discrete $L^2$-norm \\
\mbox{  }\\[-0.25cm]
\hline
& \\[-0.25cm]
$\tilde a_h(\cdot,\cdot)   $  standard FE    &  3.85104e-15\\
 $a_{h;{\mathcal I}} (\cdot,\cdot)   $, ${\mathcal I} = {\mathcal W} $ &  2.10913e-15\\
    $a_{h;{\mathcal I}} (\cdot,\cdot)   $, ${\mathcal I} = {\mathcal W}_V \cup
{\mathcal W}_E  $            &  1.75099e-15\\
      $a_{h;{\mathcal I}} (\cdot,\cdot)   $, ${\mathcal I} = {\mathcal W}_V  \cup{\mathcal W}_F $         &  6.82596e-04\\
     $a_{h;{\mathcal I}} (\cdot,\cdot)   $, ${\mathcal I} = {\mathcal W}_E \cup{\mathcal W}_F $        &  9.19145e-07\\
   $a_{h;{\mathcal I}} (\cdot,\cdot)   $, ${\mathcal I} = {\mathcal W}_V  $      &  6.82032e-04\\
   $a_{h;{\mathcal I}} (\cdot,\cdot)   $, ${\mathcal I} = 
{\mathcal W}_E  $    &  9.17523e-07\\
$a_h(\cdot,\cdot)$ stencil scaling &  6.81942e-04\\
& \\[-0.25cm] \hline
\end{tabular}}
\caption{Illustration of the effect of the hierarchical mesh structure.}
\label{fig:shift}
\end{figure}

This observation yields for a node $i \not\in {\mathcal W}_V \cup  {\mathcal W}_E$
that the sum over the nodes can be split into two parts
\begin{align*}
  a_{h;{\mathcal I}} (v_{\text{aff}}, \phi_i) = a_h (v_{\text{aff}}, \phi_i) &=
    \frac 12 \sum_{ j \in {\mathcal W}} (k(x_j) + k(x_i)) (v_{\text{aff}} (x_j) - v_{\text{aff}} (x_i))
    (\hat s_{x_i}^T)_j \\ & +  \frac 12 \sum_{j \not\in {\mathcal W}} (k(x_j) + k(x_i)) (v_{\text{aff}} (x_j) - v_{\text{aff}} (x_i))
    (\hat s_{x_i}^T)_j.
\end{align*}
For the second sum on the right, we have already shown equality to the corresponding term in the bilinear form $\tilde a_h (v_{\text{aff}}, \phi_i)$.
We recall that each node $j$ in ${\mathcal W} $ such that $x_i$ and $x_j$ form an edge  has a point mirrored node $x_{j'}$, i.e., $x_i = 0.5(x_j + x_{j'})$.
For the first term on the right, we find that for a node $j \in {\mathcal W}$
 it holds $ (\hat s_{x_i}^T)_j =  (\hat s_{x_i}^T)_{j'}$.  In terms of 
$v_{\text{aff}} (x_i+ w_j) - v_{\text{aff}} (x_i) = v_{\text{aff}} (x_i) - v_{\text{aff}} (x_i+ w_{j'})$,
the first summand on the right can thus be further simplified to
\begin{align*}
  \frac 12 \sum_{j=1 }^{n_i} \left(
\frac 12 ( k(x_i + w_j) + k(x_i)  - (k(x_i + w_{j'}) + k(x_i))) (v_{\text{aff}} (x_i+ w_j) - v_{\text{aff}} (x_i)) \right)
  (\hat s_{x_i}^T)_j .
\end{align*}
Together with \eqref{eq:shift}, this yields the stated equality.
\end{proof}

To illustrate Lemma \ref{lem:repro}, we consider a 3D example on the unit cube
 $\Omega = (0,1)^3$. 
We use as solution $u(x,y,z) = -7x +y +3z$ and  $
k(x,y,z) = 2x+3y+5z+1 $ for the coefficient function.
The standard finite element solution reproduces the exact solution up to machine precision.
We test the influence of the stencil scaling approach on different sets of primitives.
In the right of Fig.~\ref{fig:shift}, we report the discretization error in the discrete $L^2$-norm for different combinations.
Note that the macro triangulation of the unit cube consists of 12 tetrahedral elements and of one non-boundary macro vertex 
at $(0.5,0.5,0.5)$.
These considerations show that the choice ${\mathcal I} = {\mathcal W}_V \cup {\mathcal W}_E $ is quite attractive. The number of stencils which have to be expensively evaluated grows only with $2^\ell$ while we still can guarantee the reproduction property for an affine solution in the case of an affine coefficient function.

\section{Cost of a work unit}
\label{sec:cost}
The stencil scaling has been introduced 
as means to reduce the cost of a WU and thus help to design less expensive and thus more efficient 
PDE solvers.
We will 
\changed{first employ}
a cost metric based
on operation count. 
While we are aware that
real run-times will be influenced 
by many additional factors, including e.g.~the quality of the implementation,
compiler settings, and various hardware details, the classic measure still provides
useful insight.
\changed{Another important aspect that we will address here is the
question of memory accesses.}
A \changed{more technical} hardware-aware performance analysis, as conducted in \cite{gmeiner2015towards}
to evaluate the efficiency on real computers, is
beyond our current scope. 
%
%
Since asymptotically the contributions from the element primitives dominate the cost,
we restrict ourselves to the study of stencils for nodes located in the interior of a macro element $T$.
\subsection{Cost for stencil scaling}
Let $\hat{s}^T_i$ be the 15-point stencil associated with the Laplacian at
an inner node $i$ of $ T$ which is independent of the node location within $T$.
Recall that the scaled stencil is given by
%
%
%
%
\begin{equation*}
s_{ij}^{T} := \frac{1}{2}\Big(\kappa(x_i) + \kappa(x_j)\Big)\hat{s}_{ij}^{T}
\enspace,\quad
\forall j\in\mathcal{N}_T(i)\setminus{\{i\}}
\enspace,\quad
s_{ii}^{T} := - \sum_{j\in\mathcal{N}_T(i)\setminus\{i\}} s_{ij}^{T}
\end{equation*}
where $\kappa(x_i)$ represents either the value of the parameter $k$ at node
$x_i$ or a modified version of it as suggested in \eqref{eqn:defKmin}. Note
that this definition is the direct translation of the approach described
in Sec.~\ref{subsec:stencilScaling} from the bilinear form to the pure stencil.
Assuming that the constant $\nicefrac{1}{2}$ factor is incorporated into the
stencil $\hat{s}_{i}^{T}$ at setup and taking into account the number of
edges emanating from $x_i$, being six in 2D and 14 in 3D, we can compute the
non-central stencil entries with 12 and 28 operations, respectively. 
The computation of the central entry via the zero row-sum property takes 5 and 13 additions,
respectively.
This is summarized in Tab.~\ref{tab:costStencilAssembly}.

We next compare the cost for the stencil scaling approach to an on-the-fly
computation based on classic FEM techniques which, however, exploits the
advantages of hierarchical hybrid grids, see also \cite{Gmeiner:2013:PhD}.

\begin{remark}
The cost to approximate the stencil coefficients with the two-scale interpolation method of
\cite{Bauer:2017:Two-Scale} amounts to $q$ operations per stencil entry 
when the polynomial degree is chosen as $q$.
When the coefficients are locally smooth, $q=2$ delivers good results
in \cite{Bauer:2017:Two-Scale}.
The cost of this method becomes
$2 \times 7=14$ operations in 2D, and $2 \times 15=30$ operations in 3D, respectively.
This is slightly cheaper than the stencil scaling of this paper,
but it neglects the setup cost that is necessary to construct the polynomials
and completely ignores the fact that for a fixed macro mesh size no asymptotic optimality can be achieved. 
\end{remark}

\subsection{Cost of the on-the-fly computation for the classical FEM}
Employing the bilinear form \eqref{eq:standardfe} we can write the stencil
coefficients as
\begin{equation}
\label{eqn:onTheFlyStencil}
\tilde{s}_{ij}^T = \tilde{a}_h ( \phi_j,\phi_i)
= \sum_{t\in\mathcal{T}_h^{e;T}} \sum_{\ell=1}^{d+1} \frac{k(x_\ell^t)}{(d+1)}
   \nabla \phi_{j_t} \cdot \nabla \phi_{i_t} |t|
= \sum_{t\in\mathcal{T}_h^{e;T}} (E^t)_{i_t,j_t} \sum_{\ell=1}^{d+1} k(x_\ell^t)
\end{equation}
where
\begin{equation*}
(E^t)_{i_t,j_t} := \frac{1}{(d+1)} \integrald{t}{}{\nabla \phi_{j_t} \cdot \nabla \phi_{i_t}}{x}
\end{equation*}
denotes the local stiffness matrix including the volume averaging factor
$\nicefrac{1}{(d+1)}$. The
regularity of the mesh inside a macro element implies that there exist only
a fixed number of differently shaped elements $t$. Thus, we can compute stencil entries
on-the-fly from one (2D) or three (3D) pre-computed stiffness matrices and the
node-based coefficient values%
\footnote{Note that in our HHG implementation, we use six element-matrices as
this is advantageous with respect to local to global indexing; see
\cite{Bergen:2006:SCS} for details.}.
In 2D, always six elements $t$ are attached to an interior node. Summing nodal values of
$k$ for each element first would require 12 operations. Together with the
fact that two elements are attached to each edge, computing all six non-central
stencil entries via \eqref{eqn:onTheFlyStencil} would then require a total of
30 operations. This number can be reduced by 
eliminating common sub-expressions. 
We pre-compute the values
\begin{equation*}
k(x_i) + k(j_2)\enspace,\quad
k(x_i) + k(j_4)\enspace,\quad
k(x_i) + k(j_6)\enspace,
\end{equation*}
i.e.,~we sum $k$ for the vertices at the ends of the edges marked \changed{as dashed} in
Fig.~\ref{fig:2DsubExpr}. 
Then, we obtain $\sum_{\ell=1}^3 k(x_\ell^t)$ by adding the value of $k$ at the
third vertex of $t$ to the pre-computed expression.
In this fashion, we can compute all six sums with 9 operations and obtain a
total of 27 operations for the non-central entries.
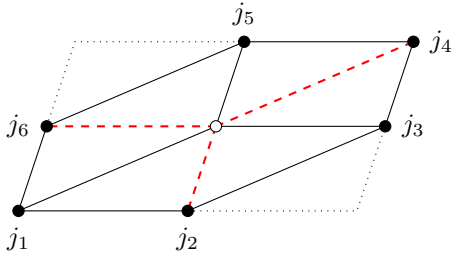
\begin{figure}[t]\centering
\begin{tikzpicture}[scale=0.75]
\draw (0.0,0.0) -- (3.0,0.0) -- (6.5,1.5) -- (7.0,3.0) --
      (4.0,3.0) -- (0.5,1.5) -- cycle;
\draw (3.5,1.5) -- (0.0,0.0);
\draw (3.5,1.5) -- (6.5,1.5);
\draw (3.5,1.5) -- (4.0,3.0);
\draw[thick,draw=red,dashed] (3.5,1.5) -- (7.0,3.0);
\draw[thick,draw=red,dashed] (3.5,1.5) -- (3.0,0.0);
\draw[thick,draw=red,dashed] (3.5,1.5) -- (0.5,1.5);
\draw[dotted] (3.0,0.0) -- (6.0,0.0) -- (6.5,1.5);
\draw[dotted] (0.5,1.5) -- (1.0,3.0) -- (4.0,3.0);
\filldraw[draw=black,fill=white] (3.5,1.5) circle (0.1);
\filldraw[draw=black,fill=black] (0.0,0.0) circle (0.1);
\filldraw[draw=black,fill=black] (3.0,0.0) circle (0.1);
\filldraw[draw=black,fill=black] (6.5,1.5) circle (0.1);
\filldraw[draw=black,fill=black] (7.0,3.0) circle (0.1);
\filldraw[draw=black,fill=black] (4.0,3.0) circle (0.1);
\filldraw[draw=black,fill=black] (0.5,1.5) circle (0.1);
%
\draw (0.0,-0.5) node {$j_1$};
\draw (3.0,-0.5) node {$j_2$};
\draw (7.0, 1.5) node {$j_3$};
\draw (7.5, 3.0) node {$j_4$};
\draw (4.0, 3.5) node {$j_5$};
\draw (0.0, 1.5) node {$j_6$};
\end{tikzpicture}
\caption{Neighborhood of node $i$ and edges [red/dashed] selected for common
sub-expression elimination.\label{fig:2DsubExpr}}
\end{figure}

The situation in 3D is 
more complicated. This stems mainly from the
fact that the number of elements sharing the edge from node $i$ to node $j$
is no longer a single value as in 2D. Instead we have in 3D that
$|\mathcal{T}_h^{e;T}|=4$ for the gray, blue, and green type edges that emanate from node $i$,
while $|\mathcal{T}_h^{e;T}|=6$ for the remaining 8 red type edges. Thus, 
for the computation of the non-central stencil entries we obtain the cost
\begin{equation*}
6 ( 4 + 3 ) + 8 ( 6 + 5 ) + 40 = 170
\enspace.
\end{equation*}
The \changed{summand} 40 represents the operations required for summing the nodal $k$
values employing again elimination of common sub-expressions; see
\cite{Gmeiner:2013:PhD} for details.
Assembling the central entry involves contributions from 6 elements in 2D and
24 in 3D. Thus, executing this via \eqref{eqn:onTheFlyStencil} cannot require less
operations than using the row-sum property.
%
%
\subsection{Cost of stencil application}
The cost of the application of the assembled stencil is independent of the 
approach to compute the stencil.
Using the first expression in
\begin{equation*}
s_{i}^Tv = \sum_{j\in\mathcal{N}_T(i)} s^T_{ij}v(x_j)
        = \sum_{j\in\mathcal{N}_T(i)\setminus{\{i\}}}{s}^T_{ij} (v(x_j) - v(x_i))
\end{equation*}
results in 13 operations in 2D and 29 in 3D. Exploiting the row-sum property
avoids computing the central entry, but increases the cost for the stencil
application, cf. second expression, so that in total only a single operation is
saved. Tab.~\ref{tab:costStencilAssembly} summarizes the results of this section.
We observe that assembling the stencil and applying it once in the scaling
approach requires \changed{in 2D only about \nicefrac{2}{3} and in 3D about \nicefrac{1}{3}} of the operations compared to a classical
on-the-fly variant.
%
%
\begin{table}[htb]\centering\footnotesize
\begin{tabular}{*{5}{c}}
\hline\noalign{\smallskip}
approach & dimension & non-central entries & central entry
& assembly + application\\
\noalign{\smallskip}\hline\noalign{\smallskip}
\multirow{2}{*}{stencil scaling} & 2D &  6 add /  6 mult &  5 add / 0 mult &  17 add / 12 mult \\
                                 & 3D & 14 add / 14 mult & 13 add / 0 mult &  41 add / 28 mult \\
\noalign{\smallskip}\hline\noalign{\smallskip}
\multirow{2}{*}{on-the-fly FEM}  & 2D & 15 add / 12 mult &  5 add / 0 mult &  26 add / 18 mult \\
                                 & 3D & 98 add / 72 mult & 13 add / 0 mult & 125 add / 86 mult \\
\noalign{\smallskip}\hline\noalign{\smallskip}
\end{tabular}
\caption{Comparison of operation count for assembling and applying the stencil
using either stencil scaling or on-the-fly assembly.}
\label{tab:costStencilAssembly}
\end{table}

In 3D, the stencil scaling approach results in a total of 69 operations per node,
and is thus only $2.4$ times more expensive than the constant-coefficient case
treated in e.g.~\cite{Bergen:2006:SCS}.
The 69 operations save a factor of 3 compared to the 211 operations per node for the conventional,
yet highly optimized on-the-fly assembly and roughly an order of magnitude when compared with
unoptimized variants of on-the-fly-assembly techniques.

To which extent a factor three savings in floating point operations
will be reflected in run-time (or other practical cost metrics, such as e.g.,~energy consumption),
depends on many details of the hardware and system software. 
For instance, we note that in the above algorithms, the
addition and multiplication operations are not ideally balanced and they do not always 
occur such that the fused multiply-add operations of a modern processor architecture
can be used. Furthermore, advanced optimizing
compilers will restructure the loops and will attempt to find a scheduling
of the instruction stream that avoids dependencies.  
Thus, in effect, the number of operations executed on the processor
may not be identical to the number of operations calculated
from the abstract algorithm or the high level source code.
Nevertheless, though the operation count does not permit a precise prediction of the run-times, we will see in the following that our effort to reduce the number of operations
pays off in terms of accelerated execution.

\changed{
\subsection{Memory accesses}
\label{subsec:memcost}

One other key aspect governing the performance of any algorithm is 
the number of read and write operations it needs to perform 
and their pattern with respect to spatial and temporal locality.
The influence of these properties results
from the large disparity between peak floating point performance of modern
CPUs and the latency and bandwidth 
limitation of memory access.
All modern architectures employ a hierarchy of caches \cite{Comer:2005:PPH,Hager:2011:CRC}
that helps accellerate memory access but that also make an a-priori 
prediction of run-times quite difficult

As in the previous subsections for the number of operations, we will perform here a
high-level analysis of memory traffic, before presenting experimental results in
Sec.~\ref{subsec:roofline}. 
As a baseline we are including in
our consideration not only the on-the-fly and the stencil scaling approach,
but also a \emph{stored stencil} version.
For the latter we assume that in a
setup phase the stencil for each node was assembled, by whatever method, and
then stored.
For the stencil application the weights of the local stencil are then loaded
from main memory and applied to the associated degrees of freedom.
Since the stencil weights correspond to the non-zero matrix entries of the
corresponding row, this would be the HHG analogue of performing a sparse
matrix-vector multiplication with a matrix stored in a standard sparse storage
format, such as e.g.~Compressed Row Storage~(CRS).
Note, however, that due to the structuredness of the mesh inside a volume
primitive, less organizational overhead and indirections are required than
for a sparse matrix format. Most importantly we do neither need to store nor
transfer over the memory system any information on the position of the non-zero
matrix entries.

For this article we do not study the full
details of algorithmic optimization for best usage of the memory sub-system.
Instead, we are going to compare two idealized cases. These are the
\emph{optimistic} version, in which we assume perfect re-use of each data
item loaded to the caches and a \emph{pessimistic} version, where there is
no re-use at all.
Any actual implementation will lie somewhere in between these two cases
and the closeness to one of them being determined by algorithmic properties and
the quality of its implementation. Furthermore, we will only consider the 3D
problem.

Common to all three approaches under consideration is that, in order to apply
the local stencil and compute the residuum at a node, they need to load the DOFs at the node and its
14 neighbors and the value of the right-hand side at the node
itself. After the stencil application, the resulting nodal value must be written
back. Thus, we are not going to inspect these parts of the stencil application
and also neglect questions of write-back strategies for the caches.

With respect to the stencil weights the situation is, of course, a different
one. Let us start with the \emph{stored stencil} approach. We denote by $N$ the
number
of DOFs inside a single volume primitive. In 3D a scalar operator using our discretization is
represented by a 15-point-stencil. This structure is invariant inside the
volume primitive due to the regular mesh structure. Hence, we obtain for the
total number of data items to be loaded from memory
$\mathcal{N}_{\text{stored}} = 15N$.

In the on-the-fly assembly we start by loading the six pre-computed element
matrices, which are 
in $\RR^{4\times 4}$ for the 3D case. Note that these are loaded only once when
the first node of the volume primitive is treated and can stay in the L1
cache during the complete loop over the volume primitive as they only
occupy 768 bytes.
Assembling the local stencil from these matrices requires information on
the coefficient function $k$. Assuming perfect re-use of the nodal $k$
information, i.e.~each nodal value needs to be loaded only once, we obtain
$\mathcal{N}^{\text{opt}}_{\text{fly}} = 6\times (4\times 4) + N = 96 + N$. In the
pessimistic case, where we assume absolutely no cache effects, we need to
reload neighboring $k$ values each time we update another node. This gives us
$15N$ load operations per node resulting in 
$\mathcal{N}^{\text{pes}}_{\text{fly}} = 96 + 15N$.

Finally we consider the stencil scaling approach. Here it is sufficient to load
once the 14 non-central weights of the reference stencil $\hat{s}$ for the
volume primitive%
\footnote{The current HHG implementation for technical reasons assembles the
stencil from the pre-computed element matrices also in this case, so the
constant term in the memory access is the same 768 bytes as for the
on-the-fly approach.}.
These are 112 bytes and the values can, as in the on-the-fly
approach, remain in the L1 cache. The non-central weights are then scaled
depending on the neighboring $k$ values, while its central weight is derived
using the zero-sum property.
This gives us $\mathcal{N}^{\text{opt}}_{scale} = 14 + N$ and 
$\mathcal{N}^{\text{pes}}_{scale} = 14 + 15N$.
\begin{table}[!t]\footnotesize
\centering
\caption{\changed{Cost of the three different approaches measured in total bytes loaded
for computation of stencil weights, assuming the use of the IEEE binary64
data-type.}}\label{tab:LoadValues}
\changed{\begin{tabular}{lcc}
\hline\noalign{\smallskip}
approach & optimistic & pessimistic \\
\noalign{\smallskip}\hline\noalign{\smallskip}
stored stencils  &   $120N$ &     $120N$ \\
on-the-fly FEM   & $768+8N$ & $768+120N$ \\
stencil scaling  & $112+8N$ & $112+120N$ \\
\noalign{\smallskip}\hline
\end{tabular}}
\end{table}

Table \ref{tab:LoadValues} sums up our results. Comparing the entries for the
stored stencil approach to the pessimistic bounds for the other two approaches
we find the same pre-factors for $N$ plus small constant terms. One should
note, however,
that in the case of the stored stencils approach there will be no temporal
cache effects when we proceed from one node to the next, since the stencil
values/row entries cannot be re-used. In the two other approaches, we expect
to see positive cache effects due to temporal and spatial re-use of some
values of the coefficient function $k$, i.e.~results closer to
$\mathcal{N}^{\text{opt}}$. This can be seen in the comparisons
in Sec.~\ref{subsec:roofline}.

%
}
\section{Numerical accuracy study and run-time comparison}
\label{sec:num}
In this section, we provide different numerical results which illustrate the accuracy and run-time of the new scaling approach in comparison to the element-wise
finite element assembling based on nodal integration within a matrix free framework.
We consider different cases such as scalar and tensorial coefficient functions $k$ and the scenario of a geometry mapping.
Throughout this section, we denote the time-to-solution by tts and by relative
tts always mean the ratio of the time-to-solution of the stencil scaling
approach with respect to the nodal integration.
From our theoretical considerations for one stencil application from Tab.~\ref{tab:costStencilAssembly}, we expect a relative tts of roughly one third.
Further, we denote the estimated order of convergence by eoc and the asymptotic convergence rates of the multigrid solver by $\rho$ defined as $\rho = \left(r^{(i^\ast)}/r^{(5)}\right)^{\nicefrac{1}{i^\ast-5}}$ where $r^{(i)}$ is the $L^2$ residual at iteration $i$, and $i^\ast$ the final iteration of the solver.
Each of the following 3D computations was conducted on SuperMUC Phase 2 using the Intel 17.0 compiler together with the Intel 2017 MPI library. For all runs, we specify the compiler flags \texttt{-O3 -march=native -xHost}. Note that the serial runs using only a single compute core are not limited to run on large machines like SuperMUC but can also be run on usual modern desktop workstations with enough memory. The peak memory usage by our largest serial run was at about 4.46\,GiB.

\subsection{A quantitative comparison in 2D for a scalar permeability}
In this example, we consider as domain the unit-square $\Omega=(0,1)^2$ and use a non-polynomial manufactured solution
\begin{equation*}
u(x,y) = \frac{x^4y}{xy + 1}
\enspace.
\end{equation*}
We employ as coefficient function
	$k(x,y;m) = 2+\sin(m\pi x) \sin(m\pi y)$ with $m\in \{2,4,8\}$.
The right-hand side is
computed by inserting the above definitions into the equation.
This construction has the advantage that we can
study the effect of the magnitude of $\nabla k$ in a systematic fashion by
adjusting $m$.
We perform a study on a regular triangular mesh comparing errors of the
discrete solutions obtained by two standard Galerkin finite element approaches
and our proposed scaling approach.
\changed{The error norms are approximated using a 5th order quadrature
rule, while we use a 2nd order scheme to evaluate the weak right-hand side.}
Results are listed in Tab.~\ref{tab:2d-regular-mesh}.
Here, $u_h^{\rm n}$ denotes the approximation obtained by employing a nodal quadrature rule resulting in bilinear form \eqref{eq:standardfe},
while $u_h^{\rm b}$ uses a quadrature rule that evaluates $k$ at the triangle's barycenter.
{\setlength{\tabcolsep}{.4em}
\begin{table}[!b] \small
\centering\footnotesize
\begin{tabular}{c|c|c|c|c|c|c|c|c|c|c|c|c}
\hline
&\multicolumn{4}{c|}{midpoint integration} & \multicolumn{4}{c|}{nodal integration} & \multicolumn{4}{c}{stencil-approach}\\
\hline
\rule[-1.5mm]{0pt}{4.75mm}$L$ & $\|u - u_h^{\rm b}\|_0$ & eoc & $|u - u_h^{\rm b}|_1$ & eoc & $\|u - u_h^{\rm n}\|_0$ & eoc & $|u - u_h^{\rm n}|_1$ & eoc & $\|u - u_h^{\rm s}\|_0$ & eoc & $|u - u_h^{\rm s}|_1$ & eoc \\
\hline
\multicolumn{13}{c}{$m=2$}\\
\hline
	0 & 3.75e-03 & 0.00 & 9.13e-02 & 0.00 & 3.25e-03 & 0.00 & 9.15e-02 & 0.00 & 3.34e-03 & 0.00 & 9.14e-02 & 0.00\\
	1 & 9.61e-04 & 1.96 & 4.60e-02 & 0.99 & 8.59e-04 & 1.92 & 4.61e-02 & 0.99 & 8.65e-04 & 1.95 & 4.61e-02 & 0.99\\
	2 & 2.42e-04 & 1.99 & 2.31e-02 & 1.00 & 2.18e-04 & 1.97 & 2.31e-02 & 1.00 & 2.18e-04 & 1.98 & 2.31e-02 & 1.00\\
	3 & 6.06e-05 & 2.00 & 1.15e-02 & 1.00 & 5.50e-05 & 1.99 & 1.15e-02 & 1.00 & 5.48e-05 & 2.00 & 1.15e-02 & 1.00\\
	4 & 1.51e-05 & 2.00 & 5.77e-03 & 1.00 & 1.37e-05 & 2.00 & 5.78e-03 & 1.00 & 1.37e-05 & 2.00 & 5.78e-03 & 1.00\\
\hline
\multicolumn{13}{c}{$m=4$}\\
\hline
	0 & 3.61e-03 & 0.00 & 9.25e-02 & 0.00 & 3.64e-03 & 0.00 & 9.32e-02 & 0.00 & 3.54e-03 & 0.00 & 9.36e-02 & 0.00\\
	1 & 9.14e-04 & 1.98 & 4.62e-02 & 1.00 & 1.02e-03 & 1.83 & 4.70e-02 & 0.99 & 9.40e-04 & 1.91 & 4.67e-02 & 1.00\\
	2 & 2.30e-04 & 1.99 & 2.31e-02 & 1.00 & 2.81e-04 & 1.87 & 2.33e-02 & 1.02 & 2.43e-04 & 1.95 & 2.31e-02 & 1.01\\
	3 & 5.77e-05 & 2.00 & 1.15e-02 & 1.00 & 7.27e-05 & 1.95 & 1.15e-02 & 1.01 & 6.13e-05 & 1.98 & 1.15e-02 & 1.00\\
	4 & 1.44e-05 & 2.00 & 5.78e-03 & 1.00 & 1.83e-05 & 1.99 & 5.78e-03 & 1.00 & 1.53e-05 & 2.00 & 5.78e-03 & 1.00\\
\hline
\multicolumn{13}{c}{$m=8$}\\
\hline
	0 & 4.06e-03 & 0.00 & 1.01e-01 & 0.00 & 4.71e-03 & 0.00 & 1.07e-01 & 0.00 & 4.47e-03 & 0.00 & 1.03e-01 & 0.00\\
	1 & 9.16e-04 & 2.15 & 4.87e-02 & 1.06 & 1.14e-03 & 2.04 & 5.09e-02 & 1.08 & 1.09e-03 & 2.04 & 5.23e-02 & 0.98\\
	2 & 2.32e-04 & 1.98 & 2.35e-02 & 1.05 & 4.02e-04 & 1.51 & 2.55e-02 & 1.00 & 3.29e-04 & 1.73 & 2.44e-02 & 1.10\\
	3 & 5.97e-05 & 1.96 & 1.16e-02 & 1.02 & 1.21e-04 & 1.73 & 1.20e-02 & 1.08 & 8.80e-05 & 1.90 & 1.17e-02 & 1.06\\
	4 & 1.50e-05 & 1.99 & 5.78e-03 & 1.01 & 3.20e-05 & 1.92 & 5.84e-03 & 1.04 & 2.24e-05 & 1.97 & 5.80e-03 & 1.02\\
\hline
\noalign{\smallskip}
\end{tabular}
\caption{\label{tab:2d-regular-mesh} Results for a regular triangular mesh. Here, $L$ denotes the (uniform) refinement
level and eoc the estimated order of convergence.}
\end{table}}

\subsection{A quantitative comparison in 3D for a scalar permeability}
As a second test, we consider a non-linear solution similar to the 2D tests in the previous section
\begin{equation*}
u = \frac{x^3y+z^2}{xyz+1},
\end{equation*}
with a parameter dependent scalar coefficient function
	$k(x,y,z;m) = \cos(m\pi xyz)+2$
on the unit-cube $\Omega = (0,1)^3$ discretized by six tetrahedra.
\changed{Here, and in the second 3D example below, the weak right-hand side is computed by interpolating
the right-hand side associated with our manufactured solution into our finite element ansatz space and subsequent
multiplication with the mass matrix. The $L^2$ error is approximated by a discrete version, i.e.~an appropriately
scaled nodal $L^2$ error. While more advanced approaches could be used here, the ones chosen are completely
sufficient for our purpose, which is to demonstrate that our stencil scaling approach behaves analogously to a
classical FE approach.}
We employ a multigrid solver with a V(3,3) cycle on a single compute core and stop after 10 multigrid iterations, i.e.,~$i^\ast = 10$, which is enough to reach the asymptotic regime.
In Tab.~\ref{tab:hhg_nonlinear}, we report convergence of the discretization error 
for the three approaches, i.e.,~classical FE with nodal integration, stencil scaling on volumes and faces, but classical
FE assembly on edges and vertices, and stencil scaling on all primitives.
The refinement is given by $L$ where $L=-2$ denotes the macro mesh. We observe quadratic convergence of the discrete $L^2$ error for all three approaches and a relative tts of about 32\% on level $L=6$.
\begin{table}[!t]
  \centering\footnotesize
  \begin{tabular}{c|c|c|c|c|c|c|c|c|c|c|c|c}
    \hline
    \multicolumn{2}{c|}{} & \multicolumn{3}{c|}{nodal integration} & \multicolumn{4}{c|}{scale Vol+Face} & \multicolumn{4}{c}{scale all} \\
    \hline
    L & DOF & error & eoc & $\rho$ & error & eoc & $\rho$ &
    \rule[-2mm]{0pt}{5.5mm}{\begin{minipage}{5mm}\centering rel.\\[-2pt]tts\end{minipage}}
    & error & eoc & $\rho$ &
    \rule[-2mm]{0pt}{5.5mm}{\begin{minipage}{5mm}\centering rel.\\[-2pt]tts\end{minipage}}\\
    \hline
    \multicolumn{13}{c}{$m=3$}\\
    \hline
    $1$ & 3.43e+02 & 2.46e-03 & --   & 0.07 & 2.42e-03 & --   & 0.06 & 1.13 & 2.51e-03 & --   & 0.07 & 0.70\\
    $2$ & 3.38e+03 & 7.06e-04 & 1.80 & 0.13 & 5.97e-04 & 2.02 & 0.12 & 0.45 & 6.05e-04 & 2.05 & 0.12 & 0.45\\
    $3$ & 2.98e+04 & 1.80e-04 & 1.97 & 0.16 & 1.46e-04 & 2.03 & 0.15 & 0.34 & 1.47e-04 & 2.05 & 0.15 & 0.28\\
    $4$ & 2.50e+05 & 4.46e-05 & 2.01 & 0.18 & 3.59e-05 & 2.02 & 0.15 & 0.30 & 3.59e-05 & 2.03 & 0.15 & 0.29\\
    $5$ & 2.05e+06 & 1.11e-05 & 2.01 & 0.17 & 8.88e-06 & 2.01 & 0.14 & 0.31 & 8.88e-06 & 2.02 & 0.14 & 0.32\\
    $6$ & 1.66e+07 & 2.75e-06 & 2.01 & 0.16 & 2.21e-06 & 2.01 & 0.13 & 0.32 & 2.21e-06 & 2.01 & 0.13 & 0.32\\
    \hline
    \multicolumn{13}{c}{$m=8$}\\
    \hline
    $1$ & 3.43e+02 & 3.17e-03 & --   & 0.07 & 5.43e-03 & --   & 0.09 & 1.33 & 5.21e-03 & --   & 0.08 & 0.70 \\
    $2$ & 3.38e+03 & 1.50e-03 & 1.08 & 0.15 & 1.54e-03 & 1.82 & 0.15 & 0.53 & 1.56e-03 & 1.74 & 0.15 & 0.45 \\
    $3$ & 2.98e+04 & 4.83e-04 & 1.63 & 0.19 & 4.04e-04 & 1.93 & 0.17 & 0.31 & 4.06e-04 & 1.94 & 0.16 & 0.36 \\
    $4$ & 2.50e+05 & 1.31e-04 & 1.89 & 0.19 & 1.01e-04 & 1.99 & 0.16 & 0.30 & 1.02e-04 & 2.00 & 0.17 & 0.29 \\
    $5$ & 2.05e+06 & 3.32e-05 & 1.98 & 0.20 & 2.52e-05 & 2.01 & 0.16 & 0.31 & 2.53e-05 & 2.01 & 0.17 & 0.32 \\
    $6$ & 1.66e+07 & 8.30e-06 & 2.00 & 0.21 & 6.29e-06 & 2.01 & 0.16 & 0.33 & 6.29e-06 & 2.01 & 0.16 & 0.32 \\
    \hline
    \noalign{\smallskip}
  \end{tabular}
  \caption{\label{tab:hhg_nonlinear} 3D results in the case of a scalar coefficient function with errors measured in the discrete $L^2$-norm.}
\end{table}

\subsection{A quantitative comparison in 3D for a permeability tensor}
\label{sub:generalM}
As a third test, we consider a full symmetric and positive definite permeability tensor $K$ with non-linear components. The off-diagonal components are negative or zero:
\begin{equation*}
K = \begin{pmatrix}x^{2} + 2 y^{2} + 3 z^{2} + 1 & - y^{2} & - z^{2}\\- y^{2} & 2 x^{2} + 3 y^{2} + z^{2} + 1 & - x^{2}\\- z^{2} & - x^{2} & 3 x^{2} + y^{2} + 2 z^{2} + 1\end{pmatrix}
\end{equation*}
The manufactured solution is  set to
\begin{equation*}
u = \frac{x^4y+2z}{xyz+1} ,
\end{equation*}
and we consider as domain the unit-cube discretized by twelve tetrahedra. We employ the same multigrid solver as in the previous subsection on a single compute core, but stop the iterations if the residual is reduced by a factor of $10^{-9}$. We denote the final iteration by $i^\ast$.
The results for the classical FE approach and the stencil scaling approach on volumes and faces with classical FE assembly on edges and vertices are reported in Tab.~\ref{tab:divK4}. We observe quadratic convergence of the discrete $L^2$ error for both approaches and a relative tts of 31\% on our finest level $L=6$.
\begin{table}[!b]
  \centering\footnotesize
  \begin{tabular}{c|c|c|c|c|c|c|c|c}
    \hline
    & & \multicolumn{3}{c|}{nodal integration} & \multicolumn{3}{c|}{scale Vol+Face} & rel.  \\
    L & DOF & error & eoc & $\rho$ & error & eoc & $\rho$ & tts \\
    \hline
    $1$ &  8.55e+02  &  1.92e-03  &  --    &  0.07  &  2.21e-03  &  --    &  0.07  &  0.60\\
    $2$ &  7.47e+03  &  4.40e-04  &  2.13  &  0.16  &  5.19e-04  &  2.09  &  0.16  &  0.35\\
    $3$ &  6.26e+04  &  1.06e-04  &  2.05  &  0.24  &  1.27e-04  &  2.03  &  0.24  &  0.25\\
    $4$ &  5.12e+05  &  2.60e-05  &  2.02  &  0.30  &  3.15e-05  &  2.01  &  0.30  &  0.26\\
    $5$ &  4.15e+06  &  6.47e-06  &  2.01  &  0.35  &  7.86e-06  &  2.00  &  0.35  &  0.30\\
    $6$ &  3.34e+07  &  1.61e-06  &  2.00  &  0.37  &  1.96e-06  &  2.00  &  0.37  &  0.31\\
    \hline
    \noalign{\smallskip}
  \end{tabular}
  \caption{\label{tab:divK4} 3D results in the case of a tensorial permeability with discretization errors measured in the discrete $L^2$-norm.}
\end{table}
\changed{Note that, as in Tab.~\ref{tab:hhg_nonlinear}, the relative tts exhibits a slightly
non-monotonic behaviour. This stems from the fact that different types of primitives (e.g.~faces and
volume) have a different tts behaviour for increasing $L$ and profit differently from the scaling
approach. For large $L$, tts is dominated by the work performed on the volume DOFs as is the
relative tts between two approaches. A test with a single macro tetrahedron showed a monotonic
behaviour for the relative tts.}

\changed{\subsection{Memory traffic and roofline analysis}
\label{subsec:roofline}

In Sec.~\ref{subsec:memcost} a theoretic assessment of memory accesses was
presented. In order to verify these results we devised a benchmark to run on a
single compute node of SuperMUC Phase~2. We compare our proposed scaling
approach against the on-the-fly assembly and the stored stencil variant,
described in Sec.~\ref{subsec:memcost}.
Floating-point performance and memory traffic are measured using the Intel
Advisor 2018 tool \cite{intel-advisor}.

We start by giving a brief summary of the SuperMUC Phase 2 hardware details.
Values were taken from \cite{SuperMUC:Configuration:Details}.
One compute node consists of two Haswell Xeon E5-2697 v3 processors clocked at
2.6\,GHz. Each CPU is equipped with 14 physical cores. Each core has a dedicated
L1 (data) cache of size 32\,kB and a dedicated L2 cache of size 256\,kB. The
theoretical bandwidths are 343\,GB/s and 92\,GB/s, respectively. The CPUs
are running in cluster-on-die mode. Thus, each node represents four NUMA
domains each consisting of 7 cores with a separate L3 cache of size 18\,MB and
a theoretical bandwidth of 39\,GB/s.
%
%
Note that the cache bandwidth scales linearly with the number of cores.
Furthermore, each node provides 64\,GB of shared memory with a theoretical
bandwidth of 6.7\,GB/s.

The benchmark computes the residual $r=f-A(k)u$, for a scalar
operator $A(k)$. We only consider volume primitives and their associated DOFs
in the benchmark. The residual computation is iterated 200 times to improve
signal to noise ratio. The program is executed using 28 MPI processes, pinned to the 28 physical cores of a single node. 
This is essential to avoid overly optimistic bandwidth values when only a single core executes memory accesses.
Measurements with the Intel Advisor are carried out solely on rank 0. Moreover, all measurements are
restricted to the inner-most loop, i.e.~where the actual nodal updates take
place, to obtain a clear picture. 
This restriction does not influence the results as the outer loops are
identical in all three variants.
We choose $L=6$ as refinement level, which gives us $2.7\cdot10^6$ DOFs per
MPI rank. The computation involves three scalar fields $u$, $k$, and $f$ which
each require $\sim$22\,MB of storage.
In Fig.~\ref{fig:roofline} (left) we present a roofline analysis, see
\cite{Ilic:2013:CAL,Williams2009}, based on the measurements. The abscissa
shows the arithmetic intensity, i.e.~the number of FLOPs performed divided by
the number of bytes loaded into CPU registers per nodal update. The ordinate
gives the measured performance as FLOPs performed per second. The diagonal
lines give the measured DRAM and cache bandwidths. The values for the caches
are those reported by the Intel Advisor. Naturally these measured values are
smaller than the theoretical ones given in the hardware description above. 
In order to assess the practical
DRAM bandwidth of the complete node we used the LIKWID tool, see
\cite{Treibig:2010}, and performed a \textit{memcopy} benchmark with
non-temporal stores. We found that the system can sustain a bandwidth of about
104\,GB/s, which for perfect load-balancing on the nodes, as is the case in
our benchmark, results in about 3.7\,GB/s per core.
The maximum performance for double precision vectorized fused multiply-add
operations is also reported by the Advisor tool (35.7\,GFLOPs/s).

\begin{figure}\centering
\begin{tabular}{@{}c@{}c@{}}
\includegraphics[width=0.495\textwidth]%
{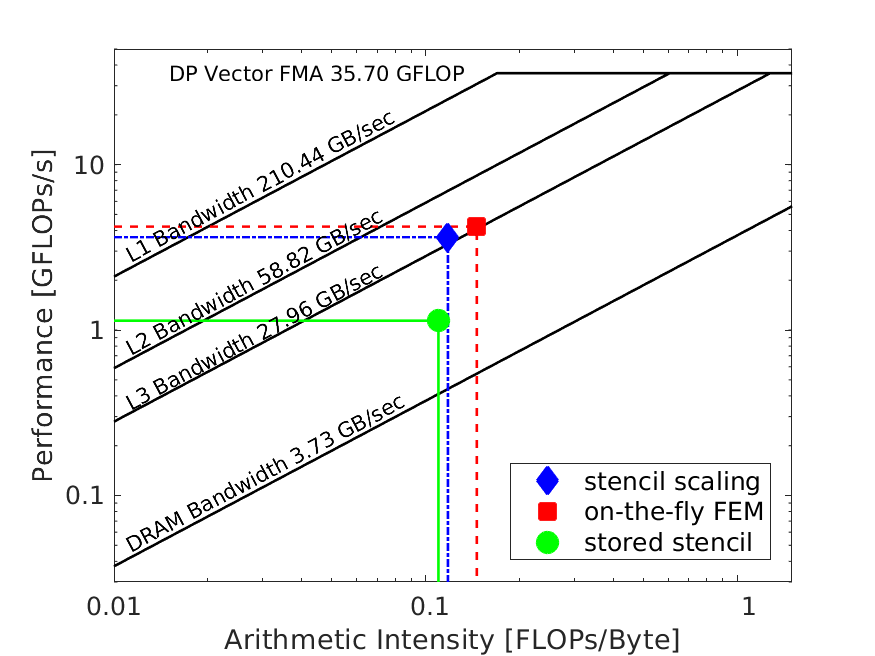} &
\includegraphics[width=0.495\textwidth]%
{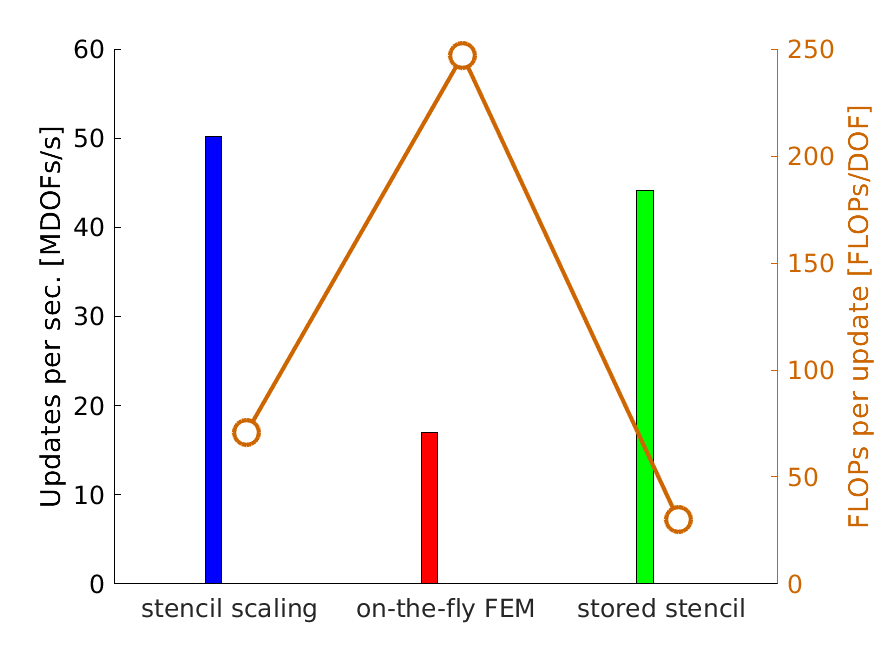} \\
\end{tabular}
\caption{\changed{Roofline model (left) and nodal updates per unit time (right) for the
three different approaches.} 
\label{fig:roofline}}
\end{figure} 

For all three kernels we observe a quite similar arithmetic intensity. The
value for the stored stencil variant is close to what one would expect
theoretically. In this case one needs to load 31 values of size 8 bytes (15
$u$ values, 1 $f$ value and 15 stencil weights), write back one residual value of 8 bytes
and perform 30 FLOPs, which results in an intensity of around 0.12 FLOPs/Byte.
Both, the stencil scaling approach and the on-the-fly assembly show a slightly
higher intensity, but the difference seems negligible.

Using FLOPs/s as a performance criterion, we see that the stored stencil
approach achieves significantly smaller values than the two other
approaches,
which perform about a factor of four better. The limiting resource for both,
the on-the-fly assembly and the stencil scaling, is the L3 bandwidth. This
confirms our expectations from Sec.~\ref{subsec:memcost} showing that values
of $k$ are not evicted from L3 cache and can be re-used.

For the stored stencil approach the performance is somewhere in between the
limits given by the L3 and DRAM bandwidths. Here, all stencil values have to be
loaded from main memory, while values of $u$ may still be kept in the cache, 
as with the two other approaches. Due to overlapping of computation and load
operations, the performance is still above the DRAM bandwidth limit, but clearly below
what the L3 bandwidth would allow.

Further assessment requires a more detailed analysis like the
execution-cache-memory model \cite{Hager2012} which goes beyond the scope of
this article.

From the application point of view FLOPs/s is not the most relevant criterion,
however. Of main importance to a user is overall run-time. In this respect the
interesting measure is the number of stencil applications per second, or
equivalently the number of DOFs updated per second (DOFs/s). The latter can
be derived from the FLOPs/s value based on the number of operations required
to (assemble and) apply a local stencil for the three different methods. These
values and the derived DOFs/s are shown in the right part of
Fig.~\ref{fig:roofline}. The number of FLOPs required per update
was computed from the number of performed FLOPs, as reported by the
Intel Advisor, and the number of DOFs inside a volume primitive. Note that
these values match very well our theoretical considerations from
Tab.~\ref{tab:costStencilAssembly}.

As the FLOPs/s value attained by the on-the-fly and the stencil scaling
approaches are almost identical, it is the reduced number of operations
required in the stencil scaling variant, which directly pays off.
The three times lower FLOP count directly translates to a threefold increase in
the DOFs/s and, thus, a similar reduction in run-time.

Furthermore, we observe that our stencil scaling version also gives a higher
number of DOFs/s than the stored stencil approach. While this increase is not
as dramatic compared to the on-the-fly assembly, we emphasize
that the stored stencil approach within HHG does not require additional memory
traffic for information on the matrix' sparsity pattern or involve in-direct
accesses like in a classical CRS format. More importantly, for the largest
simulation with $L=6$, carried out in the Sec.~\ref{subsec:applicationBlending}
below, the stored stencil approach over all levels of the mesh hierarchy would
require about 22\,TB of storage. Together with the scalar fields $u$, $f$ and
the tensor $K$ this would exceed the memory available on SuperMUC. Per core
typically 2.1\,GB are available to an application,
\cite{SuperMUC:Configuration:Details}, which sums up to around
30\,TB for the 14\,310 cores used. The stencil scaling approach on the other
hand only requires to store 120 bytes, i.e.~a single 15-point-stencil per
primitive and level, for the operator construction. This results in less than
165\,MB over all cores for the complete simulation.
This value could even be further reduced to 26\,MB exploiting similarities
of stencils on different levels and the zero-sum property. However, this does
not seem worth the effort.
}

\subsection{Application to a blending setting and large scale results}
\label{subsec:applicationBlending}
To demonstrate the advantages of our novel scaling approach also for a more realistic scenario,
we consider an example using a blending function, as mentioned
in Sec.~\ref{sec:frame}.
To this end, we consider a half cylinder mantle
with inner radius $r_1 = 0.8$ and outer radius $r_2=1.0$, height $z_1 = 4.0$ and
with an angular coordinate between $0$ and $\pi$ as our physical domain $\Omega_{\text{phy}}$.
The cylinder mantle is additionally warped inwards by $w(z) = 0.2\sin\left(z\pi/z_1\right)$ in axial direction.
The mapping $\Phi:\Omega_{\text{phy}} \rightarrow \Omega$ is given by
\begin{equation*}
\Phi(x,y,z) = \begin{pmatrix} \sqrt{x^2+y^2} + w(z) \\
\arccos{\left(\nicefrac{x}{\sqrt{x^2+y^2}}\right)} \\ z
\end{pmatrix}
\end{equation*}
with the reference domain $\Omega = (r_1,r_2) \times (0,\pi) \times (0,z_1)$.
%
%
Using \eqref{eqn:aWithBlendingFunc}, it follows for the mapping tensor $K$
\begin{equation*}
K = \frac{(D \Phi)(D \Phi)^{\top}}{ | \text{det } D \Phi |} = \sqrt{x^2+y^2} \begin{pmatrix}
w'(z)^2 + 1 & 0 & w'(z) \\
0 & \nicefrac{1}{x^2+y^2} & 0 \\
w'(z) & 0 & 1 \end{pmatrix}.
\end{equation*}
Obviously, this tensor is symmetric and positive definite. 
In addition to the geometry blending, we use a variable material parameter
$
a(x,y,z) = 1 + z$.
On the reference domain $\Omega$ this yields the PDE
$-\div a K \nabla u = f$.
As analytic solution on the reference domain we set
\begin{equation*}
u(\xh,\yh,\zh) = \sin\left(\frac{\xh-r_1}{r_2-r_1} \pi \right) \cos\left(4\yh\right) \exp\left(\nicefrac{\zh}{2}\right).
\end{equation*}
The analytic solution mapped to the physical domain is 
illustrated in the right part of Fig.~\ref{fig:cylinder}.
\begin{figure}[!t]
\centering
\includegraphics[trim=215pt 20pt 280pt 80pt, clip=true, width=0.2\textwidth]{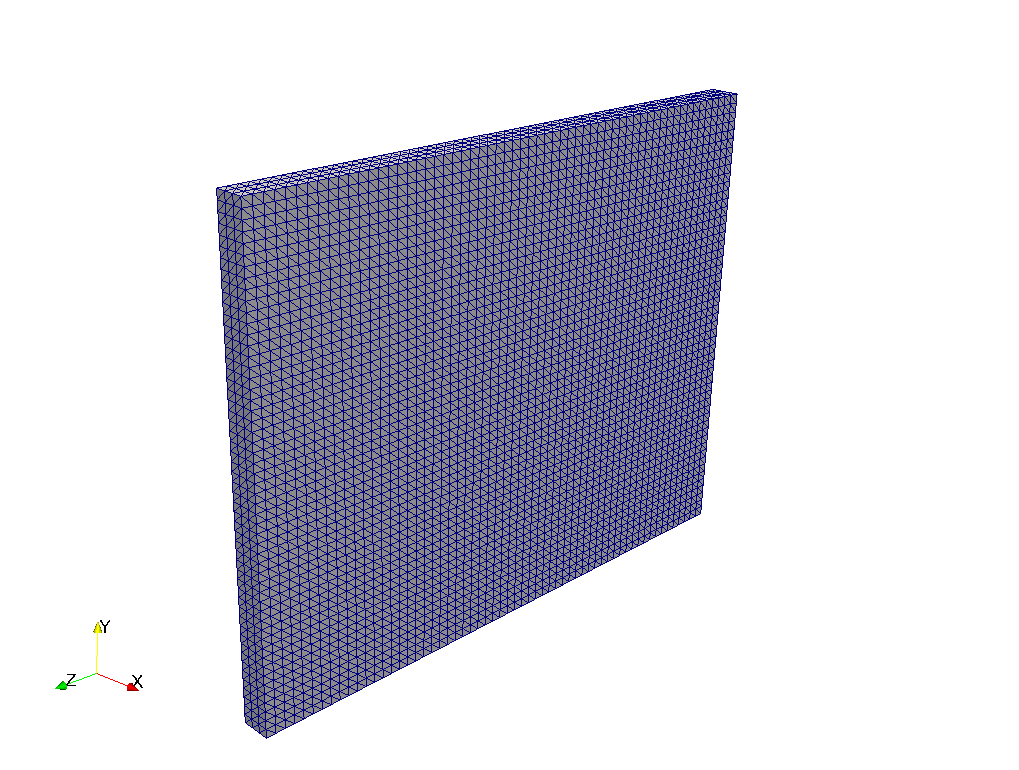}
\hspace*{1.5cm}
\includegraphics[trim=250pt 70pt 260pt 20pt, clip=true, width=0.2\textwidth]{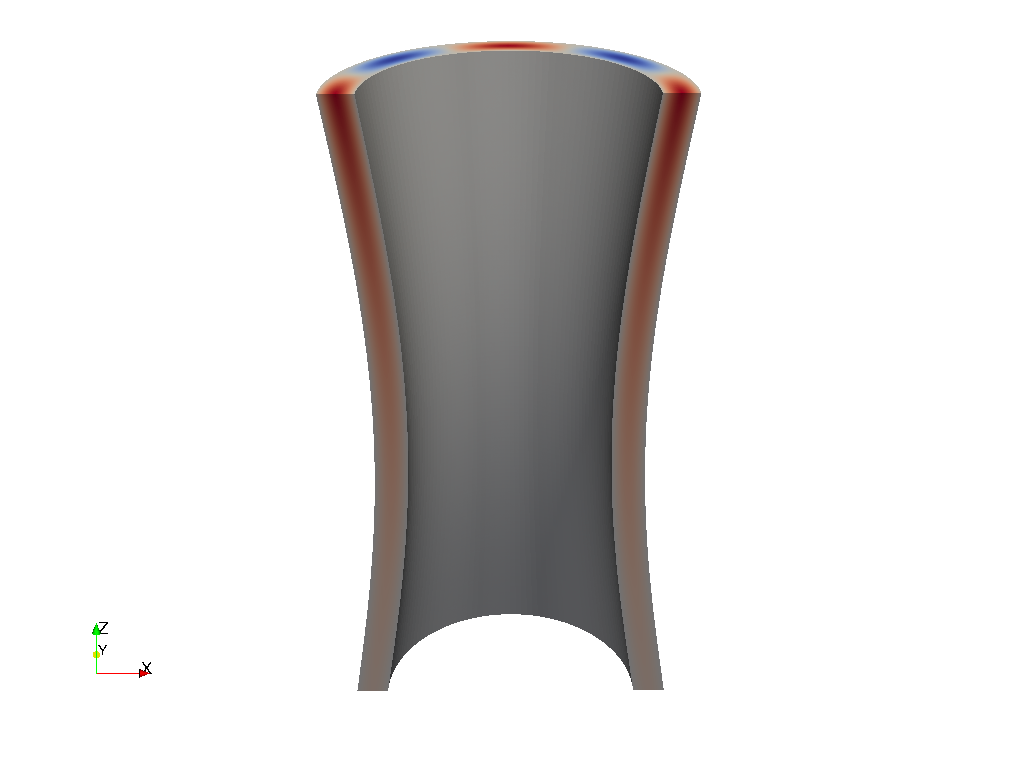}
\caption{Macro mesh of the reference domain $\Omega$ (left) and analytic solution~$u$ (right) mapped to the physical domain $\Omega_{\text{phy}}$.}
\label{fig:cylinder}
\end{figure}

For our numerical experiments, we employ a macro mesh composed of 9540 hexahedral blocks,
where each block is further split into six tetrahedral elements; see Fig.~\ref{fig:cylinder} (left).
The resulting system is solved using 14\,310 compute cores, i.e.,~we assign four macro elements per core. 
For the largest run we have a system with $\mathcal{O}\!\left(10^{11}\right)\!$
DOF.
We employ a multigrid solver with a V(3,3) cycle.
The iteration is stopped when the residual has been reduced by a factor of $10^{-8}$.
In Tab.~\ref{tab:cylinder}, we report the resulting discretization error, the asymptotic multigrid convergence order $\rho$, and the time-to-solution.

\begin{table}[!t]
\centering\footnotesize
\begin{tabular}{*5{c|}r|*3{c|}r|c}
\hline
& & \multicolumn{4}{c|}{nodal integration} & \multicolumn{4}{c|}{scale Vol+Face} & rel. \\
L & DOF & error & eoc & $\rho$ & tts & error & eoc & $\rho$ & tts & tts  \\
\hline
1 & 4.7e+06 & 2.43e-04 &  -   & 0.522 & 2.5    & 2.38e-04 &  -   & 0.522 & 2.0   & 0.80 \\
2 & 3.8e+07 & 6.00e-05 & 2.02 & 0.536 & 4.2    & 5.86e-05 & 2.02 & 0.536 & 2.6   & 0.61 \\
3 & 3.1e+08 & 1.49e-05 & 2.01 & 0.539 & 12.0   & 1.46e-05 & 2.01 & 0.539 & 4.5   & 0.37 \\
4 & 2.5e+09 & 3.72e-06 & 2.00 & 0.538 & 53.9   & 3.63e-06 & 2.00 & 0.538 & 15.3  & 0.28 \\
5 & 2.0e+10 & 9.28e-07 & 2.00 & 0.536 & 307.2  & 9.06e-07 & 2.00 & 0.536 & 88.9  & 0.29 \\
6 & 1.6e+11 & 2.32e-07 & 2.00 & 0.534 & 1822.2 & 2.26e-07 & 2.00 & 0.534 & 589.6 & 0.32 \\
\hline
\noalign{\smallskip}
\end{tabular}
\caption{Results for large scale 3D application with errors measured in the discrete $L^2$-norm. 
   \label{tab:cylinder}}
\end{table}

These results demonstrate that the new scaling approach
maintains the discretization error, as is expected on structured grids 
from our variational crime analysis, 
as well as the multigrid convergence rate.
For small $L$ we observe that the influence of vertex and edge primitives is more pronounced as the improvement in time-to-solution is only small. But for increasing
$L$ this influence decreases and for $L\geqslant4$ the run-time as compared to the nodal integration approach is reduced to about
30\%.

\subsection*{Acknowledgments}
The authors gratefully acknowledge the Gauss Centre for Supercomputing (GCS) 
for providing computing time on the supercomputer SuperMUC at
Leibniz-Rechenzen\-trum (LRZ).

\bibliographystyle{siamplain}
\bibliography{literature}

\end{document}